\documentclass[a4paper,11pt]{amsart}
\pdfoutput=0
\usepackage{amscd,epsfig,amssymb,latexsym,amsmath,url,bm}
\usepackage{enumerate}
\usepackage{amsthm}		
\usepackage{mathrsfs}
\usepackage{array}
\usepackage[all]{xy}
\usepackage{graphicx}
\usepackage{color}

\newcommand{\Zpart}{\mathsf{Z}} 
\renewcommand{\epsilon}{\varepsilon}
\newcommand{\invcirc}{\mathcal B} 
\newcommand{\loc}{{\rm loc}} 
\newcommand{\myline}{L} 
\newcommand{\mycirc}{S} 
\newcommand{\ball}{\Lambda} 
\newcommand{\denseset}{\Sigma} 
\newcommand{\mybundle}{N} 
\newcommand{\first}{z} 
\newcommand{\second}{{w}} 
\newcommand{\bsecond}{{\bm w}} 
\newcommand{\secondscalar}{{w}} 
\newcommand{\firstinfinity}{\zeta}

\newcommand{\bsecondinfinity}{\bm \omega}
\newcommand{\fixone}{\eta_{\text{\tiny{1}}}} 
\newcommand{\fixtwo}{\eta_{\text{\tiny{2}}}}
\newcommand{\const}{K} 

\newcommand{\nbhd}{\Omega} 
\newcommand{\WW}{\mathcal W} 
\newcommand{\bidisk}{X} 
\newcommand{\ra}{{real analytic}} 
\newcommand{\stable}{\mathcal W^s_{\loc}(\invcirc)}
\newcommand{\R}{\mathcal R}
\newcommand{\Rphys}{\mathcal R}
\newcommand{\C}{\mathbb{C}}
\newcommand{\K}{C}
\newcommand{\dist}{\mbox{dist}}
\newcommand{\T}{\mathbb{T}}
\newcommand{\Cphys}{\mathcal C}
\newcommand{\fp}{\eta'}
\newcommand{\fptwo}{\eta}
\newcommand{\blowup}{w}
\newcommand{\one}{\lambda}
\newcommand{\two}{\tau}
\def\B0{{\mathbf{0}}}
\newcommand{\FF}{\mathcal F}
\newcommand{\LL}{{\mathcal L}}
\newcommand{\LLzero}{\LL_0}
\newcommand{\LLone}{\LL_1}

\newcommand{\seq}{x_k}
\newcommand{\subseq}{x_{k_j}}
\newcommand{\vseq}[1]{x_{k,#1}^v}
\newcommand{\hseq}[1]{x_{k,#1}^h}
\newcommand{\vsubseq}[1]{x_{k_j,#1}^v}
\newcommand{\hsubseq}[1]{x_{k_j,#1}^h}

\newtheorem{theorem}{Theorem}[section]
\newtheorem*{theorem1}{Theorem A}
\newtheorem*{theorem1prime}{Theorem A'}
\newtheorem*{theorem2}{Theorem B}
\newtheorem*{theorem2prime}{Theorem B'}
\newtheorem{lemma}[theorem]
{Lemma}
\newtheorem{proposition}[theorem]
{Proposition}
\newtheorem{corollary}[theorem]
{Corollary}

\newtheorem*{rea}{Rea's Theorem in Codimension 1}

\theoremstyle{definition}
\newtheorem{definition}{Definition}

\theoremstyle{remark}
\newtheorem*{remark}{Remark}

\title{Superstable manifolds of invariant circles and co-dimension 1 B\"ottcher functions}
\subjclass{}
\begin{author}[S. Kaschner]{Scott R. Kaschner}
\email{srkaschn$@$iupui.edu}
\address{ %
IUPUI Department of Mathematical Sciences\\
LD Building, Room 257\\
402 North Blackford Street\\
Indianapolis, Indiana 46202-3267\\
 United States }
\end{author}
\begin{author}[R. K. W. Roeder]{Roland K. W. Roeder}
\email{rroeder@math.iupui.edu}
\address{ %
IUPUI Department of Mathematical Sciences\\
LD Building, Room 224Q\\
402 North Blackford Street\\
Indianapolis, Indiana 46202-3267\\
 United States }
\end{author}

\begin{document}

\begin{abstract}
Let $f: X~\dashrightarrow~X$ be a dominant meromorphic self-map, where $X$ is a compact, connected complex manifold of dimension $n >
1$.  Suppose there is an embedded copy of $\mathbb{P}^1$ that is invariant
under $f$, with $f$ holomorphic and transversally superattracting with degree
$a$ in some neighborhood.  Suppose $f$ restricted to this line is given by $z
\mapsto z^b$, with resulting invariant circle $S$.  We prove that if $a \geq
b$, then the local stable manifold $\WW^s_\loc(S)$ is real analytic.
In fact, we state and prove a suitable localized version that can be useful in wider contexts.
We then show that the condition $a \geq b$ cannot be relaxed without adding additional hypotheses by presenting two examples with $a < b$ for which $\WW^s_\loc(S)$ is not real analytic in the neighborhood of any point.
\end{abstract}

\maketitle

\section{Introduction.}


Let $f: X \dashrightarrow X$ be a dominant meromorphic self-map of a compact, 
connected complex manifold $X$ of dimension $n > 1$.  Here, the focus is on
the situation in which there is $\myline \subset X$, an embedded copy of $\mathbb{P}^1$,
with $f$ holomorphic in a neighborhood of $\myline$, $\myline$ is invariant, 
and $f|\myline$ is conjugate to $z \mapsto z^b$. 
We also assume $\myline$ is  transversally
superattracting of degree $a$, that is, the local coordinates of $f$ transverse
to $\myline$ vanishes to order $a$.  This is described more precisely at the beginning of $\S2$.
Although this is a rather special situation, it has appeared in examples from
\cite{roeder2005,hruska_roeder2008,blr1,blr2}.  For such maps, the Julia set of
$f|\myline$ is an invariant circle $\mycirc$, which is a hyperbolic set for
$f$.  The local stable manifold $\WW^s_\loc(\mycirc)$ is a real $2n-1$
dimensional manifold.   We will prove:

\begin{theorem1}
\label{THM1} 
If $a \geq b$, then $\WW^s_\loc(\mycirc)$ has real analytic regularity.
\end{theorem1}

To prove the theorem, we will localize to the situation to a tubular
neighborhood $\mybundle$ of $\myline$ which is forward invariant under $f$.
Theorem A is a direct consequence of the following:

\begin{theorem1prime}
\label{THM1_PRIME}
Let $\mybundle$ be a complex manifold with $\dim(\mybundle) \geq 2$, containing an embedded projective line $\myline$.  Suppose
$f \colon \mybundle \rightarrow \mybundle$ a dominant holomorphic map,
$\myline$ is invariant and transversally superattracting with degree $a$, 
and $f|\myline$ is conjugate to $z \mapsto z^b$, having invariant circle
$\mycirc$.  If $a \geq b$, then $\WW^s_\loc(\mycirc)$ has real analytic
regularity.
\end{theorem1prime}

For a diffeomorphism, the existence and regularity of the local stable manifold
for a hyperbolic invariant manifold $\mybundle$ has been studied extensively
Hirsch-Pugh-Shub in \cite{HPS}.  A strong form of hyperbolicity known as 
{\em normal hyperbolicity} is assumed in order to guarantee a $C^1$ local stable 
manifold.  Specifically, $\mybundle$ is called normally hyperbolic for $f$ if the
expansion of $Df$ in the unstable direction transverse to $\mybundle$ dominates the 
maximal expansion of $Df$ tangent to $\mybundle$ and the contraction of $Df$ in the 
stable direction transverse to $\mybundle$ dominates the maximal contraction of 
$Df$ tangent to $\mybundle$; see \cite[Theorem 1.1]{HPS}.  For $C^r$ regularity, 
there is an analogous condition in terms of the $r$-th power of the maximal 
expansion/contraction tangent to $\mybundle$.

Although the maps considered in this paper are many-to-one, they also do not
fit in the context of \cite{HPS} since $f|\myline$ is conformal, forcing that
the rates of expansion tangent to $\mycirc$ and transverse to $\mycirc$ are
equal.  Thus, $\mycirc$ is not normally hyperbolic.

In \S \ref{SEC:PROOF} we prove Theorem A' by constructing a semi-conjugacy
between $f$ and $z\mapsto~z^b$ on a forward invariant neighborhood of
$\mycirc$.  The construction is similar to the proof of the well-known B\"ottcher's Theorem from one-dimensional complex dynamics \cite{BOTTCHER}; see also \cite[Ch.~9]{milnor}.  While B\"ottcher's Theorem refers to a holomorphic change of coordinate (often called a B\"ottcher coordinate) defined in the neighborhood of a superattracting fixed point, the function we construct here is neither a coordinate, nor is it defined in a full neighborhood of a superattracting fixed point.  However, by analogy, we call it a ``co-dimension 1  B\"ottcher function.''  

Those interested in the mathematical legacy of B\"ottcher should see
\cite{STAWISKA}.  We will now briefly describe variants of B\"ottcher's Theorem
in higher dimensions.  It was shown by Hubbard and Papadopol in
\cite{hubbard_papadopol} that a B\"ottcher coordinate in higher dimension
cannot exist in general.  With additional hypotheses, their existence has been
proved in \cite[Theorem 3.2]{UEDA} and \cite{USHIKI_BOTT}.  A more detailed
criterion for existence of a B\"ottcher coordinate is presented in \cite{BEK}.
The related problem of conjugating a polynomial endomorphism to its highest
degree terms in a neighborhood of the hyperplane at infinity is studied in
\cite[Theorem 9.3]{hubbard_papadopol}, \cite[Theorem 7.4]{BEDFORD_JONSSON},
\cite{BEDFORD_JONSSON_POTENTIAL}, and \cite[Theorem 1]{PENG}.  These authors
prove that such a conjugacy exists on the stable set of the Julia set at
infinity, so long as it satisfies suitable hyperbolicity.  More recent studies
of superattracting behavior appear in \cite{FAVRE_JONSSON,GIGNAC,GIGNAC_RUGGIERO,RUGGIERO}.

The proof of Theorem A' is followed by \S \ref{SEC:EXAMPLES}, where we provide applications to certain specific examples, including those from \cite[\S 6.2]{hruska_roeder2008} and \cite{roeder2005}. 

In \S \ref{SEC:NON_ANALYTIC}, we show that the condition that $a \geq b$ cannot
be improved without adding additional hypotheses.  We'll consider two maps for
which $a < b$ and $\WW^s_\loc(\mycirc)$ is not analytic.  One of them is the
Migdal-Kadanoff renormalization map $R$ for the Ising model on the Diamond
Hierarchical Lattice (DHL) that was studied extensively in \cite{blr1,blr2}.
It has $a=2$ and $b=4$.  The other is a polynomial skew product with $a=2$ and
$b=3$.

Let us comment a bit more on the map $R$.  For this map, the invariant
circle $S$ has the physical context of being related to the bottom of the Lee-Yang
cylinder, so it is denoted $B$.  In \cite[Lemma 3.2]{blr2}, the authors proved that
$\WW^s_\loc(B)$ is a $C^\infty$ manifold.  We prove:

\begin{theorem2}
\label{THM2} 
The stable manifold $\WW^s_\loc(B)$ is not \ra~at any point.
\end{theorem2}

The proof of this theorem divides into four main parts.  First we construct a co-dimension 1 B\"ottcher function $\varphi$ defined in a neighborhood of $B$ under the assumption that $\WW^s_\loc(B)$ is real analytic. Next we extend the domain of $\varphi$ to a neighborhood of the set obtained from $\myline$ by removing the two superattracting fixed points. After that, we develop local properties of $R$ near one of these superattracting fixed points.  Lastly, we
examine the behavior of $\varphi$ and $R$ in the extension, from which we derive a contradiction.

This theorem is of physical interest, since $\WW^s_\loc(B)$ is related to phase transitions of the Ising model on the DHL at low temperatures; see \cite{blr1,blr2}.  In \S \ref{SEC:NON_ANALYTIC_LY}, we'll explain how Theorem B relates to the limiting distribution of Lee-Yang and Lee-Yang-Fisher zeros at low temperatures.

To summarize, the organization of the paper is as follows.  Section \ref{SEC:PROOF} is devoted to the proof of Theorem A', and Section \ref{SEC:EXAMPLES} describes several examples in which Theorem A can be applied.  The description of examples for which $\stable$ is not real analytic and proof of Theorem B then follow in Section \ref{SEC:NON_ANALYTIC}.  Lastly, a physical interpretation of Theorem B as related to the Ising model on the DHL at low temperatures is given in Section \ref{SEC:NON_ANALYTIC_LY}.

\vspace{0.2in} {\bf Acknowledgments.} We are very grateful to Turgay Bayraktar, Pavel Bleher,
Laura DeMarco, Jeffrey Diller, John Hubbard, Sarah Koch, Mikhail Lyubich,
Rodrigo Perez, and Amie Wilkinson for very helpful discussions.  Hubbard taught us the basic
technique that we used in the proof of Theorem A.  
Discussions with Bleher and Lyubich about the Migdal-Kadanoff renormalization motivated us for this project and helped us to realize that such a stable manifold could potentially not be real analytic.  Finally, we thank the referee for his or her very helpful comments and suggestions.

The work of Kaschner was supported partially by a Department of Education GAANN
grant 42-945-00 and partially by NSF GK-12 grant DGE-0742475. The work of
Roeder was supported partially by start-up funds from the IUPUI Department of
Mathematical Sciences and partially by NSF grant DMS-1102597.

\section{Proof of Theorem A'.}\label{SEC:PROOF}

The manifold $\mybundle$ can be described by two systems of
locally trivializing coordinates $(\first,\bsecond) \in \C \times
\mathbb{C}^{n-1}$ and $(\firstinfinity,\bsecondinfinity) \in \C \times
\mathbb{C}^{n-1}$.  For $z\neq 0$, they are related by $\firstinfinity = 1/\first$ and $\bsecondinfinity =
A_z \bsecond$, with $A_z:\mathbb{\C}^{n-1} \rightarrow  \mathbb{C}^{n-1}$ a linear isomorphism
depending holomorphically on $z$.
Let us choose these trivializations so that the dynamics on the zero section is $z \mapsto z^b$.

We will make use of standard multi-index notation.  Given $ {\bm c} \in
\mathbb{Z}_{+}^{n-1}$ and $\bsecond \in \C^{n-1}$, $\bsecond^{\bm c} =
\secondscalar_1^{c_1} \secondscalar_2^{c_2}\cdots
\secondscalar_{n-1}^{c_{n-1}}$ and $|\bm c| = c_1+\cdots+c_{n-1}$.
We will always use the standard Hermitian norm $|\bsecond| = (|\secondscalar_1|^2+\cdots+|\secondscalar_{n-1}|^2)^{1/2}$ on $\C^{n-1}$.

We have assumed $\myline$ is transversally superattracting of degree $a$.
Specifically, this means that if $\chi$ is any holomorphic function at some
point $\eta\in\myline$, vanishing along $\myline$, then for any point
$\xi\in\myline$ with $f(\xi) = \eta$, the holomorphic function $\chi\circ f$ at
$\xi$ vanishes to order at least $a$ along $\myline$.

\begin{lemma}\label{fns} There are holomorphic functions ${\bm g}_1$ and ${\bm g}_{\bm c}$ for each $|c| = a$ such that in the $(\first,\bsecond)$ coordinates
\[f(\first,\bsecond)\ =\ \left(\first^b+\bsecond \cdot {\bm g}_1(\first,\bsecond), \sum_{|\bm c| = a} \bsecond^{\bm c} {\bm g}_{\bm c}(\first,\bsecond)\right).\]
Similarly, there are holomorphic functions ${\bm h}_1$ and ${\bm h}_{\bm c}$ for each $|c| = a$ such that in the $(\firstinfinity,\bsecondinfinity)$ coordinates
\[f(\firstinfinity,\bsecondinfinity)\ =\ \left(\firstinfinity^b+\bsecondinfinity \cdot {\bm h}_1(\firstinfinity,\bsecondinfinity), \sum_{|\bm c| = a} \bsecondinfinity^{\bm c} {\bm h}_{\bm c}(\firstinfinity,\bsecondinfinity)\right).\]
\end{lemma}

\begin{proof}
The proof is the same in both coordinate systems, so we'll work in the $(\first,\bsecond)$ system.  Since $f\mid\myline$ is the map $\first\mapsto \first^b$, the first coordinate of $f$ minus $\first^b$ vanishes on $\myline$.  Since $\myline$ is given by $\bsecond = \bm 0$, we have that the first coordinate of $f$ is $\first^b+\bsecond \cdot {\bm g}_1(\first,\bsecond)$ for some holomorphic function ${\bm g}_1$.  Meanwhile, the expression for the second coordinate follows from the fact that $\myline$ is transversally superattracting of degree $a$.
\end{proof}

\subsection{Hyperbolic theory.}

We'll now verify that the local stable manifold $\WW^s_\loc(\mycirc)$ is a $2n-1$ real-dimensional
topological manifold that is foliated by local stable manifolds of each point of $\mycirc$.

The hyperbolic theory for endomorphisms is somewhat less standard than for
diffeomorphisms.  Suitable references from the context of complex dynamics
include \cite{BEDFORD_JONSSON,FS_hyperbolic,jonsson}.  For consistency, we will use definitions and results from \cite[Appendix B]{BEDFORD_JONSSON}.
Let us consider the natural extension
\begin{eqnarray*}
\hat{\mycirc} &:=& \{(x_i)_{i \leq 0} \, : \, x_i \in \mycirc \mbox{ and } f(x_i) = x_{i+1}\}.
\end{eqnarray*}
We'll denote such histories by $\hat x = (x_i)_{i \leq 0} \in \hat{\mycirc}$.
Notice that the action of $f$ naturally lifts to an action $\hat f: \hat{\mycirc} \rightarrow \hat{\mycirc}$.


\begin{lemma}\label{hyperbolic} $\mycirc$ is a hyperbolic set for the map $f$.
\end{lemma}
\begin{proof}
Note that for $x\in\mycirc$, we have
\[Df_x\ =\ \left[\begin{array}{cc}b\first^{b-1}
&\displaystyle\frac{\partial}{\partial \bsecond}{\bm g}_1(\first,\bm0)\\\bm0&\bm0\end{array}\right].\]
Thus, we have $E^s(x)=\ker(Df)$ and $E^u(\hat x)\subset\myline$, so $T_x\mathbb
C^n=E^s(x)\oplus E^u(\hat x)$.  Invariance of $E^s(x)$ follows from the
fact any point in the kernel is collapsed to $(0,\bm 0)$ under $Df$, and
invariance of $E^u(\hat x)$ follows from the invariance of $\myline$.
Also, for any $v^s\in E^s(x)$ and $v^u\in
E^u(\hat x)$ with $n\geq 0$,
\[\begin{array}{l}
\|Df_x^nv^s\|=0\leq C\lambda^n\|v^s\|\mbox{ and }
\|Df_x^{n}v^u\|\leq C\lambda^{-n}\|v^u\|,
\end{array}\]
for $C =1$ and $\lambda = 1/2$.
Thus, we have that $\mycirc$ is hyperbolic.
\end{proof}

Therefore, by the stable manifold theorem (see, for example, \cite[Theorem 
5.2]{pugh1989ergodic}) each point $x \in \mycirc$
will have local stable manifold $\WW^s_{\rm loc}(x) $ that is a
complex $n-1$ ball holomorphically embedded into $\mybundle$ and each
history $\hat x$ will have a local unstable manifold $\WW^u_{\rm loc}(\hat
x)$, which is a holomorphic disc.  They depend continuously on $x$ and $\hat x$.
(In this case, the unstable manifolds all lie in $\myline$.)

\begin{remark}
Existence of such stable laminations has also been proved in the holomorphic
context by Ushiki \cite{USHIKI_SADDLE}.  It can be proved in the following
simple way as well, which is a direct generalization of what was done in
\cite[Proposition 4.2]{roeder2005} and \cite[Proposition 9.2]{blr1}.

By the stable manifold theorem for a point (see, for example, \cite[\S
2.6]{palis1982} or \cite{USHIKI_BOTT}, which hold even if $Df$ has an
eigenvalue of $0$), there exists a local stable manifold, $\WW^s_\loc((1,\bm
0))$, which is the graph of a holomorphic function $z = \eta_1(\bsecond)$
defined on some $(n-1)$-dimensional open ball, $\ball$, in the $\bsecond$ axis.
Let $\denseset\subset\mycirc$ to be the set of iterated preimages of $(1,\bm
0)$.  Using a suitable invariant cone field and a well-chosen neighborhood of
$\mycirc$, one can take iterated preimages of $\WW^s_\loc((1,\bm 0))$ so that
the preimage through each $x \in \denseset$ is expressed as the graph of a
holomorphic function $\eta_x(\bsecond)$ defined on $\ball$, making $\ball$
smaller if necessary.  In this way, we can construct local stable manifolds
over $\denseset$, which is dense in $\mycirc$.  The function
$\eta\colon\ball\times\denseset\rightarrow\mathbb C$ given by
$\eta(\bsecond,x)=\eta_x(\bsecond)$ defines a holomorphic motion of
$\denseset\subset\mathbb C$, parameterized by $\bsecond\in\ball\subset\mathbb
C^{n-1}$.  We may use the $\lambda$-lemma \cite{MSS,Lyubich} to extend $\eta$
continuously to a holomorphic motion of $\overline{\denseset}=\mycirc$,
obtaining stable manifolds for every point of $\mycirc$.
\end{remark}

\begin{definition}
A hyperbolic set $\hat\Lambda$ has a local product structure, if $\delta>0$ can be chosen small enough so that for any $p\in\Lambda$ and $\hat q\in\hat\Lambda$, either $\WW^s_{\delta}(p)\cap\WW^u_{\delta}(\hat q)$ is empty or it is a single point $x\in\Lambda$ so the unique history $\hat x$ of $x$ satisfying $x_j\in\WW^u_{\delta}(\hat f^j(\hat q))$ for all $j\leq0$ is completely contained in $\hat\Lambda$.
\end{definition}

\begin{lemma}\label{m} $\mycirc$ has local product structure for the map $f$.
\end{lemma}
\begin{proof}
By Lemma \ref{hyperbolic}, $S$ is hyperbolic.  Recall that for any $\hat q\in\hat\mycirc$, we have that $\WW^u_{\delta}(\hat q)=\mathbb D_{\delta}(q_0)\subset \myline$, which is the disc of radius $\delta>0$ centered at the point $q$ contained in $\myline$.  Since $\WW^u_{\delta}(\hat q)$ depends only on $q_0$, existence of a local product structure for $\hat\mycirc$ is very simple.

By the Stable Manifold Theorem, we may choose $\delta>0$ small enough so that for any $p\in S$, we have $\WW_{\delta}^s(p)\cap L=\{p\}$.  Thus, for any two points $p,q\in S$, the intersection $\WW^s_{\delta}(p)\cap\WW^u_{\delta}(\hat q)=\{p\}$, with $p\in S$.  Moreover, $p$ has a unique history $\hat p=(p_i)_{i\leq0}$ with $p_j\in\WW^u_{\delta}(\hat f^j(\hat q))$ for all $j\leq0$, and it is completely contained in $\hat S$ as well.
\end{proof}

Given a neighborhood $\Omega$ of $\mycirc$, let
\begin{equation}\label{s}
\WW^s_\loc(\mycirc) := \{x \in \mybundle \colon f^n x\in \Omega\mbox{ and } f^n x\rightarrow\mycirc\mbox{ as }n\rightarrow\infty\}
\end{equation}
(where $\Omega$ is implicit in the notation, and an assertion involving
$\WW^s_\loc(\mycirc)$ means that it holds for any sufficiently small
neighborhood of $\mycirc$). 

Since $\mycirc$ has a local product structure $\WW^s_\loc(\mycirc)$ is the
union of the local stable manifolds $\WW^s_{\loc}(x)$ of points $x \in
\invcirc$; see \cite[Proposition B.6]{BEDFORD_JONSSON}.  The local stable manifolds of points are pairwise disjoint and depend continuously on the base point, therefore we have:

\begin{corollary}\label{COR:MAN} $\WW^s_\loc(\mycirc)$ is a topological manifold of real dimension $2n-1$.
\end{corollary}

Note that up to this point, we have not made use of the assumption that $a\geq b$.

\subsection{Co-dimension 1 B\"ottcher function.}\label{SEC:BOTTCHER}

Let $(\first_n,\bsecond_n) := f^n(\first,\bsecond)$. 
Motivated by B\"ottcher's theorem \cite{BOTTCHER},\cite[p. 86]{milnor}, we consider a sequence of functions
\begin{eqnarray*}
\varphi_n(\first,\bsecond) = \first_n^{1/b^n}.
\end{eqnarray*}
We will show that the $\varphi_n$ converge uniformly on compact subsets of some forward invariant neighborhood $\Omega$ of $\mycirc$ to a holomorphic function $\varphi$ that semi-conjugates $f$ to $z \mapsto z^b$: 
\begin{eqnarray}\label{BOTTCHER_INVARIANCE}
\varphi(f(\first,\bsecond)) = \varphi(\first,\bsecond)^b.
\end{eqnarray}

To make sense of the $b^n$-th roots and the limit, we'll rewrite each  $\varphi_n$ as telescoping product:
\begin{equation}\label{INFINITE_PRODUCT}
\varphi\ =\ \lim_{n\rightarrow\infty}\varphi_n\ =\ {\first_0}\cdot\frac{{\first_1}^{1/b}}{{\first_0}}\cdot\frac{{\first_2}^{1/b^2}}{{\first_1}^{1/b}}\cdot\frac{{\first_3}^{1/b^3}}{{\first_2}^{1/b^2}}\cdots\ =\ {\first_0}\prod_{n=0}^{\infty}\left(\frac{{\first_{n+1}}}{{\first_n}^{b}}\right)^{\frac{1}{b^{n+1}}},
\end{equation}
where it follows from Lemma \ref{fns} that
\begin{eqnarray}\label{QUOTIENT1}
\frac{{\first_{n+1}}}{{\first_n}^{b}} = \frac{{\first_n^b+\bsecond_n \cdot {\bm g}_1(\first_n,\bsecond_n)}}{{\first_n}^{b}} = 1+\frac{\bsecond_n}{{\first_n}^{b}}\cdot {\bm g}_1(\first_n,\bsecond_n).
\end{eqnarray}
In the $(\firstinfinity,\bsecondinfinity)$ coordinates we have:
\begin{eqnarray}\label{QUOTIENT2}
\frac{{\first_{n+1}}}{{\first_n}^{b}} = \frac{\firstinfinity_n^b}{\firstinfinity_{n+1}} = \frac{1}{1+\frac{\bsecondinfinity_n}{\firstinfinity_n^b}\cdot {\bm h}_1(\firstinfinity_n,\bsecondinfinity_n)}.
\end{eqnarray}
When working in $\WW^s(\fixone)$ we'll use expression (\ref{QUOTIENT1}), when
working in $\WW^s(\fixtwo)$ we'll use expression (\ref{QUOTIENT2}), and when
working on $\WW^s_\loc(\mycirc)$, we'll use either.

We'll construct a forward invariant neighborhood $\Omega$ of $\mycirc$ so that
if $(\first,\bsecond) \in \Omega \cap (\WW^s(\fixone) \cup \WW^s_\loc(\mycirc))$, then 
\begin{eqnarray}
\left|\frac{\bsecond_n}{{\first_n}^{b}}\cdot {\bm g}_1(\first_n,\bsecond_n)\right| &<& \frac{1}{2} \label{NEED1}, 
\end{eqnarray}
and if $(\firstinfinity,\bsecondinfinity) \in \Omega \cap (\WW^s(\fixtwo) \cup \WW^s_\loc(\mycirc))$, then 
\begin{eqnarray}
\left|\frac{\bsecondinfinity_n}{\firstinfinity_n^b}\cdot {\bm h}_1(\firstinfinity_n,\bsecondinfinity_n)\right| &<& \frac{1}{2}. \label{NEED2}
\end{eqnarray}
Then, 
for points in $\Omega$, the $b^n$-th root is defined by taking the branch cut along the negative real axis.  Moreover, this condition will also imply convergence of the infinite product (\ref{INFINITE_PRODUCT}) on $\Omega$, since the corresponding sum of logarithms converges:
\begin{equation*}
\sum_{n=1}^{\infty}\log\left|\frac{{\first_{n+1}}}{{\first_n}^{b}}\right|^{\frac{1}{b^{n+1}}} \leq \sum_{n=1}^{\infty} \frac{1}{b^{n+1}} \log 2. 
\end{equation*}

To construct $\Omega$, first note that for any $\const_1 > 0$ sufficiently
small, $\{|\bsecond| \leq \const_1 \} \cap \left(\WW^s(\fixone) \cup
\WW^s_\loc(\mycirc)\right)$ is a compact subset of $\C^n$.  Since ${\bm g}_1$
is holomorphic on $\C^n$, there is a bound $|{\bm g}_1(\first,\bsecond)| \leq
\const_2$ on any such compact set.  A similar bound holds in the other
coordinate system.  Therefore, it suffices to show:

\begin{lemma}\label{LEM:GOOD_NBHD}
Given any $\const > 0$, there exists a forward invariant neighborhood of $\mycirc$ in which
\begin{eqnarray}\label{EQN:DESIRED_REGION}
\frac{|\bsecond|}{|\first|^b} < K  \qquad \text{and} \qquad \frac{|\bsecondinfinity|}{|\firstinfinity|^b} < K.
\end{eqnarray}
\end{lemma}

\begin{proof}
We will take an inductive  sequence of $b$ point blow-ups at each of the two fixed points $\fixone$ and $\fixtwo$.  Using the forms of $f$ given by Lemma \ref{fns}, the calculation will be the same at each of these two points, so we'll focus on $\fixone$, which is given by $(\first,\bsecond) = (0,\bm 0)$.

We first do a point blow-up at $\fixone$, producing an exceptional divisor
$E_{\fixone,1}$.  Let $\tilde{\myline}_1$ be the proper transform of $\myline$.
We then blow-up the point intersection point between $E_{\fixone,1}$ and
$\tilde{\myline}_1$, producing a new exceptional divisor $E_{\fixone,2}$ and
proper transform $\tilde{\myline}_2$.  We inductively do this $b-2$ additional
times, each time blowing up the intersection point between the previous
exceptional divisor and proper transform of $L$.

Consider the system of coordinates $\first, {\bm \lambda} = \frac{\bm
\bsecond}{\first^b}$ centered at the intersection point of $E_{\fixone,b}$ with
$\tilde{L}_b$.  Let us denote $(\first',\bm \lambda') = \tilde{f}(\first,\bm
\lambda)$, where $\tilde{f}$ is the extension of $f$ to the final blow-up.  We
have
\begin{eqnarray*}
\first' &=& \first^b + \first^b {\bm \lambda} \cdot {\bm g}(\first,\first^b \bm \lambda) \\
\bm \lambda ' &=& \frac{\bsecond'}{(\first')^b} = \frac{\Sigma_{|\bm c| = a} (\first^b {\bm \lambda})^{\bm c} {\bm g}_{\bm c}(\first,\first^b {\bm \lambda})}{(\first^b + \first^b {\bm \lambda} \cdot {\bm g}(\first,\first^b \bm \lambda))^b} = \frac{z^{b(a-b)} \Sigma_{|\bm c| = a}  {\bm \lambda}^{\bm c} {\bm g}_{\bm c}(\first,\first^b {\bm \lambda})}{(1 + {\bm \lambda} \cdot {\bm g}(\first,\first^b \bm \lambda))^b}.
\end{eqnarray*}
Notice that this extension $\tilde{f}$ is holomorphic in a neighborhood of
$(\first,\bm \lambda) = (0,\bm 0)$ and that this point is superattracting for
$\tilde{f}$. 

Therefore, for any $\epsilon_1 > 0$ and $K \geq \delta_1 > 0$,  sufficiently small,
$\widetilde{U}_1 : = \{|\first| < \epsilon_1, |\bm \lambda| < \delta_1\}$ will be forward invariant under $\tilde{f}$.   Hence,
\begin{eqnarray*}
U_1(\epsilon_1,\delta_1) := \pi\left(\widetilde{U}_1(\epsilon_1,\delta_1)\right) = \left\{|\first| < \epsilon_1, \frac{|\bsecond|}{|\first|^b} < \delta_1\right\}
\end{eqnarray*}
will be a forward invariant set for $f$.

As stated before, the same calculation can be done at $\fixtwo$, with analogous
results.  
In particular, for any $\epsilon_2 > 0$ and $K \geq \delta_2 > 0$  sufficiently small
we will have a forward invariant set for $f$ of the form
\begin{eqnarray*}
U_2(\epsilon_2,\delta_2) = \left\{|\firstinfinity| < \epsilon_2, \frac{|\bsecondinfinity|}{|\firstinfinity|^b} < \delta_2\right\}.
\end{eqnarray*}

Let $V \subset \mybundle$ be a forward invariant tubular neighborhood of
$\myline$ and let 
\begin{eqnarray*}
V(\epsilon_1,\epsilon_2) = V \setminus \left(\{|\first| < \epsilon_1\}\cup \{|\firstinfinity| < \epsilon_2\}\right).
\end{eqnarray*}
Note that if $V$ sufficiently small, then all points of $V(\epsilon_1,\epsilon_2)$ satisfy (\ref{EQN:DESIRED_REGION}).
We will show that $V$ can be made even smaller, if necessary, in order to make
\begin{eqnarray*}
\Omega := V(\epsilon_1,\epsilon_2) \cup U_1(\epsilon_1,\delta_1) \cup U_2(\epsilon_2,\delta_2)
\end{eqnarray*}
forward invariant.

\begin{figure}
\begin{center}
\scalebox{0.83}{
\input{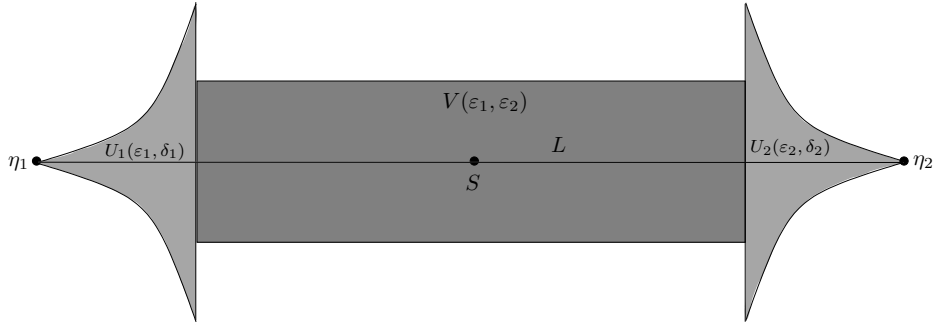}
}
\caption{The forward invariant neighborhood $\Omega$}
\end{center}
\end{figure}

\noindent
Since $U_1(\epsilon_1,\delta_1)$ and $U_2(\epsilon_2,\delta_2)$ are forward
invariant, we need only check that if $x \in V(\epsilon_1,\epsilon_2)$
and $f(x) \not \in V(\epsilon_1,\epsilon_2)$, then $f(x) \in
U_1(\epsilon_1,\delta_1) \cup U_2(\epsilon_2,\delta_2)$.  Let us focus on $x \in \WW^s(\fixone)$, since the proof will be the same
for $x \in \WW^s(\fixtwo)$.

Let $x = (\first,\bsecond) \in V(\epsilon_1,\epsilon_2) \cap \WW^s(\fixone)$ and
let $(\first_1,\bsecond_1) = f(\first,\bsecond)$.  Since $(\first,\bsecond)
\in V(\epsilon_1,\epsilon_2)$, $|\bsecond|/|\first|^b < K$, so that (\ref{NEED1})
and (\ref{QUOTIENT1}) imply that the $|\first_1| \geq |\first|^b / 2 \geq \epsilon_1^b/2$.
Thus, we need only choose the (forward invariant) tubular neighborhood $V$ sufficiently small so that
\begin{eqnarray*}
V \cap \left\{\frac{\epsilon_1^b}{2} \leq |\first| \leq \epsilon_1\right\} \subset U_1(\epsilon_1,\delta_1).
\end{eqnarray*}
Doing the same thing near $\fixtwo$, we construct a forward invariant neighborhood $\Omega$ satisfying (\ref{EQN:DESIRED_REGION}).
\end{proof}

\subsection{Completing the proof of Theorem A'}

Using the invariance (\ref{BOTTCHER_INVARIANCE}), for any $(\first,\bsecond) \in \WW^s_\loc(\mycirc)$ we have $|\varphi(\first,\bsecond)| = 1$ so that $\psi:=\log|\varphi|$ will be a real analytic function that vanishes on $\WW^s_\loc(\mycirc)$.  Notice on that $\myline$, we have $\varphi(\first,\bm 0) = \first$ and hence $\psi(\first,\bm 0) = \log|\first|$.  Since the derivative $D\psi$ is non-zero on $\mycirc$, we have that $\{\psi = 0\}$ is a real analytic $2n-1$ real-dimensional manifold in some neighborhood of $\mycirc$.

By Corollary \ref{COR:MAN}, $\WW^s(\mycirc) \subset \{\psi =
0\}$ is also a real $2n-1$ dimensional manifold.  Thus, by invariance of
domain, $\WW^s(\mycirc) = \{\psi = 0\}$ in this neighborhood.
\qed

\section{Examples illustrating Theorem A.}\label{SEC:EXAMPLES}

\subsection{Regular Polynomial Endomorphisms of $\mathbb C^2$.}

%
Suppose $f\colon\mathbb C^2\rightarrow\mathbb C^2$ is a regular polynomial endomorphism of degree  $d\geq2$.  Then $f$ has the form
\begin{equation}
f(x,y)=(p(x,y),q(x,y)),
\end{equation} 
where $p$ and $q$ are polynomials whose highest degree terms have no common zeros other than $(0,0)$, so $f$ extends holomorphically to $\mathbb P^2$. Then the line at infinity, $L_{\infty}$, is transversally superattracting with degree $d$, and $f\mid L_{\infty}$ is a one variable rational map of degree~$d$ having Julia set $J_{\infty}\subset L_{\infty}$. 
\begin{corollary}\label{COR:REG_ENDO} If $f$ is a regular polynomial endomorphism of $\mathbb C^2$ for which $f \mid L_\infty$ is conjugate to $z\mapsto z^d$, then $\WW^s_\loc(J_{\infty})$ has real analytic regularity.
\end{corollary}

Real analyticity of the stable manifold considered in \cite[\S 6.2]{hruska_roeder2008} is a direct application of Corollary \ref{COR:REG_ENDO}.

\subsection{Degenerate Newton Mappings.}

Newton mappings used to find the common roots of $P(x,y) = x(1-x)$
and $Q(x,y) = y^2+Bxy-y$  were considered dynamically in \cite{roeder2005}.  They have the form
\begin{eqnarray}\label{DEG_NORMALIZATION1}
N (x,y) = \left(\frac{x^2}{2x-1},\frac{y(Bx^2+2xy-Bx-y)}{(2x-1)(Bx+2y-1)}\right).
\end{eqnarray}
\noindent
We will consider their extension as rational maps of $\mathbb{P}^1 \times \mathbb{P}^1$.
They are skew products with the first coordinate having superattracting fixed points of degree $2$ at $x=0$ and $x=1$, so
the vertical lines $\{x=0\} \times \mathbb{P}^1$ and $\{x=1\}
\times \mathbb{P}^1$ are transversally superattracting for $N$ with the same degree.  Using the formula, one can check 
that $N$ has no indeterminate points in some neighborhood of these two lines.

Restricted to $\{x=0\} \times \mathbb{P}^1$, $N$ is the one-dimensional Newton map for the quadratic polynomial with roots at $y=0$ and $y=1$.
It is therefore conjugate to $z \mapsto z^2$, having an invariant circle $S_0$ corresponding to the points of equal distance from
$y=0$ and $y=1$ in $\mathbb{P}^1$.  ($S_0$ is the closure of ${\rm Im}(y) = \frac{1}{2}$ in $\mathbb{P}^1$.)

Similarly, the restriction of $N$ to $\{x=1\} \times \mathbb{P}^1$ is the one-dimensional
Newton map for the quadratic polynomial with roots at $y=0$ and $y=1-B$.  Thus, it is conjugate
to $z \mapsto z^2$, with an invariant circle $S_1$ corresponding to the
points of equal distance from $y=0$ and $y=1-B$ within $\mathbb{P}^1$.

Both of the lines $\{0\}\times \mathbb{P}^1$ and $\{1\} \times \mathbb{P}^1$ is
transversally superattracting with degree $2$, with the restriction of $N$
to each of them conjugate to $z \mapsto z^2$.  Therefore, it follows immediately
from Theorem A that the local stable manifolds $\WW^s_\loc(S_0)$ and
$\WW^s_\loc(S_1)$ are real analytic.  This was proven previously in
\cite{roeder2005} using more specific details of the mapping.
  
\subsection{An Example with indeterminacy.}

Consider the polynomial mapping $g\colon\C^2\rightarrow\C^2$ 
given by
\begin{eqnarray}\label{EQN:G}
g(x,y) = \left(x^2+y(1+xy),y^3(1+xy)\right).
\end{eqnarray}
Within $\C^2$, the line $L:=\{y=0\}$ is invariant and transversally
superattracting with degree $3$ and $g|L$ is given by $x \mapsto x^2$.
Let $S := \{|x| =1, y=0\}$ be the invariant circle.
Although there is the needed domination between the degrees ($3 > 2$), to apply Theorem
A we need to check how $g$ extends to a neighborhood of infinity on $L$.
The extension of $g$ to $\mathbb{P}^2$ is given in homogeneous coordinates by
\begin{eqnarray*}
g[X:Y:Z] = [X^2 Z^3+YZ^2(Z^2+XY):Y^3(Z^2+XY):Z^5].
\end{eqnarray*}
There is a point of indeterminacy for $g$ at $[1:0:0]$ on the projective line $Y=0$, which we'll also denote by $L$.
Therefore, Theorem A does not immediately apply.

Let us perform two blowups.  We first blow-up the point $[1:0:0]$ and we then
blow-up the point where the proper transform of $L$ intersects the exceptional
divisor over $[1:0:0]$.  We'll denote the space obtained after doing these two
blow-ups by $\widetilde{\mathbb{P}^2}$, the projection by $\pi:
\widetilde{\mathbb{P}^2} \rightarrow \mathbb{P}^2$, the proper transform of $L$
after these two blow-ups by $\widetilde{L}$, the invariant circle within
$\widetilde{L}$ by $\widetilde{S}$, and the lift of $g$ to the blown-up space
by $\widetilde{g}: \widetilde{\mathbb{P}^2} \rightarrow
\widetilde{\mathbb{P}^2}$.

A neighborhood of $\widetilde{L}$ can be described by two systems of coordinates
$(x,y)$ and $(\zeta,\tau)$, where $x = X/Z, y=Y/Z$ are the original affine
coordinates on $\C^2$ and $\zeta = Z/X, \tau = XY/Z^2$.  In the first system of
coordinates, $\widetilde{g}$ is given by (\ref{EQN:G}).  In the second system of coordinates, $\widetilde{g}$ is given by
\begin{eqnarray*}
\widetilde{g}(\zeta,\tau) = \left({\frac {{\zeta }^{2}}{1+\tau\,{\zeta }^{3}(1+\tau)}},{\tau}^{3}  \zeta  \left( 1+\tau \right)  \left( 1+ \tau{\zeta }^{3}+{\tau}^{2}{\zeta }^{3} \right) \right).
\end{eqnarray*}
In the second system of coordinates, $\widetilde{L}$ is given by $\tau = 0$, so we see
that $\widetilde{g}$ is holomorphic in a neighborhood of
$\widetilde{L}$.  Moreover, $\widetilde{L}$ invariant and transversally
superattracting with degree $3$ and $\widetilde{g}|\widetilde{L}$ still given
by $x \mapsto x^2$.  Therefore, Theorem A applies to give that the
local stable manifold $\WW^s_\loc(\widetilde{S})$ for $\widetilde{S}$ under $\widetilde{g}$ is real analytic.

Notice that $\widetilde{g}$ and $g$ are birationally conjugate by means of
$\pi$.  Moreover, restricted to small neighborhoods of $\widetilde{S}$ and $S$,
this birational conjugacy becomes an honest holomorphic conjugacy.  Since the
local stable manifolds  $\WW^s_\loc(\widetilde{S})$ and
$\WW^s_\loc(S)$ are defined in terms of the action of iterates of
$\widetilde{g}$ and $g$, respectively,  on these small neighborhoods, we
conclude that $\WW^s_\loc(S)$ is also real analytic.

\begin{remark}
This third example illustrates that in order to  apply Theorem A, one sometimes needs 
to do some blow-ups to obtain a map without indeterminacy in a neighborhood of $L$. 
%
\end{remark}

\section{Examples for which $\WW^s_\loc(\mycirc)$ is not real analytic.}\label{SEC:NON_ANALYTIC}

We'll now show that the hypothesis in Theorem A that $\myline$ is transversally superattracting with degree greater than
or equal to the degree of $f|\myline$ cannot be eliminated without adding additional hypotheses. 

The Migdal-Kadanoff Renormalization map $R: \mathbb{P}^2 \rightarrow \mathbb{P}^2$ for the Ising model on the DHL is given
in homogeneous coordinates by 
\begin{eqnarray*}
R[U:V:W] = [(U^2+V^2)^2:V^2(U+W)^2:(W^2+V^2)^2].
\end{eqnarray*}
For this map, the projective line $L_0 = \{V=0\}$ is transversally
superattracting with degree $2$ with $R$ holomorphic on a forward invariant
neighborhood of $L_0$.  Restricted to $L_0$, $R$ is given by $u \mapsto u^4$,
where $u=U/W$, so $a=2$ and $b=4$.  The invariant circle is denoted $B := \{V=0,|u|
= 1\}$.  Below, we will show that $\WW^s_\loc(B)$ is not real analytic in the
neighborhood of any point of $B$, thus proving Theorem B.

The second example for which $a < b$ and $W^s(S)$ is not real analytic is the following polynomial skew product of $f: \mathbb{P}^2 \rightarrow \mathbb{P}^2$ given in affine coordinates by
\begin{eqnarray*}
f(\first,\second) = (\first^3 + 2\second\first^2, \second^2).
\end{eqnarray*}
One can check that this map is holomorphic on a forward invariant neighborhood in $\mathbb{P}^2$ of the invariant line $\myline=\{\second=0\}$.
Moreover, $\myline$ is transversally superattracting with degree $2$, and $f|\myline$ is given by $\first \mapsto \first^3$.  Thus, $a = 2 < 3 =b$.
For this map, $\WW^s_\loc(\mycirc)$ is not real analytic in the neighborhood of any point of $\mycirc$.

In this section, we'll provide a detailed proof of Theorem B, showing that $\WW^s_\loc(B)$ is not real analytic.  An adaptation of
the same techniques can be used to show the analogous result for the skew product $f$.  We leave details of this adaptation to the reader.

\subsection{The Migdal-Kadanoff Renormalization.}

In the remainder of this section, we will adopt the notation from the recent
preprints \cite{blr1,blr2} by Bleher, Lyubich, and Roeder.
Although $R: \mathbb{P}^2 \rightarrow \mathbb{P}^2$ is more convenient for illustrating Theorem A, in the proof of Theorem B
it will be more convenient to work the expression of the Migdal-Kadanoff renormalization $\R: \mathbb{P}^2 \rightarrow \mathbb{P}^2$ in the physical coordinates $(z,t)$.  In these coordinates, it is given by
\begin{equation}\label{r}
(z_{n+1},t_{n+1})\ =\ \left(\frac{z_n^2+t_n^2}{z_n^{-2}+t_n^2},\frac{z_n^2+z_n^{-2}+2}{z_n^2+z_n^{-2}+t_n^2+t_n^{-2}}\right)\ :=\ \R(z_n,t_n).
\end{equation}
We consider $(z,t)$ as affine coordinates on $\mathbb{P}^2$ with $z = Z/Y, t=T/Y$ for some system of homogeneous coordinates $[Z:T:Y]$.
The map $\R$ has an invariant projective line $\LLzero = \{T=0\}$ that is transversally superattracting, except for an indeterminate point at
$\B0 := [0:0:1]$, and $\R | \LLzero$ is given by $z \mapsto z^4$.  The invariant circle is given by $\invcirc=\{|z|=1,t=0\}$. 

The map $R$ is semi-conjugate to $\R$ by means of a rational map $\Psi: \mathbb{P}^2 \rightarrow \mathbb{P}^2$:
\begin{eqnarray}
\begin{CD}
\mathbb{P}^2 @>\R>> \mathbb{P}^2 \\
@VV\Psi V        @VV\Psi V\\
\mathbb{P}^2 @>R>> \mathbb{P}^2
\end{CD}
\end{eqnarray}
with $[U:V:W] = \Psi([Z:T:Y]) = [Y^2:ZT:Z^2]$.
The map $\Psi$ sends $\LLzero$ to $\myline_0$, $\invcirc$ to $B$, and is holomorphic in a neighborhood of $\invcirc$. 
Therefore, $\WW^s_\loc(\invcirc) = \Psi^{-1}(\WW^s_\loc(B))$.  In particular, if $\WW^s_\loc(B)$ were real analytic in the neighborhood of any
point of $B$, then $\WW^s_\loc(\invcirc)$ would be real analytic in the neighborhood of the preimage of that point under $\Psi$.
So, Theorem B will follow from:

\begin{theorem2prime}
\label{THM2PRIME} 
The stable manifold $\stable$ is not \ra~at any point.
\end{theorem2prime}

\begin{remark}
The reason we originally stated Theorem B for $R$ rather than $\R$ is that $R$ is holomorphic 
in a full neighborhood of $L_0$, so that it illustrates why the hypothesis on $a$ and $b$ can't be eliminated in Theorem A.  One can also resolve the indeterminacy $\B0 \in \LLzero$ for $\R$,
placing it in the context of Theorem A,
via a suitable birational modification (two blow-ups and one blow-down), but that is somewhat more complicated.
\end{remark}

We will begin by proving the following proposition, and proof of Theorem B' will follow shortly thereafter.

\begin{proposition}\label{u} $\stable$ is not real analytic in any full neighborhood of $\invcirc$.
\end{proposition}

This proposition will be proven by contradiction, so for the remainder of this section, we assume $\stable$ is real analytic in a full neighborhood of $\invcirc$.  We will begin by describing the dynamics of $\R$ near $\LLzero$, and after that, with the construction of a co-dimension 1 B\"ottcher function $\varphi$.  This is followed by the extension of the domain of $\varphi$ and an exploration of the behavior of $\varphi$ and $\R$ in the extension.  The section concludes with a proof of Proposition \ref{u}.

\subsection{Dynamics in a Neighborhood of $\LLzero$.}

We will now briefly summarize basic properties of the dynamics for $\R$ in a neighborhood of $\LLzero$ from \cite[Section 4]{blr1}.

Let $\mathbb{D}_0 := \{|z| < 1, t=0\} \subset \LLzero$.  The orbit of any $z \in
\mathbb{D}_0$ will converge to an indeterminate point $\B0 := \{(0,0)\}$.  (Informally, we will denote these points by $\WW^s(\B0)$.)
Meanwhile, points near $\B0$ but not on $\LLzero$ will converge to a
superattracting fixed point $\fptwo := \{(0,1)\}$.

To see what happens for large $|z|$, we write $\R$ in homogeneous coordinates,
obtaining
\begin{equation}
\R\colon[Z:T:Y]\mapsto[Z^2(Z^2+T^2)^2:T^2(Z^2+Y^2)^2:(Z^2+T^2)(T^2Z^2+Y^4)].
\end{equation}
There is another superattracting fixed point $\fp := [1:0:0]$, which attracts all points of $\LLzero$ with $|z| > 1$.

\begin{lemma}
$\WW^s(\B0) \cup \WW^s_{\loc}(\fptwo) \cup \stable \cup \WW^s_{\loc}(\fp)$ fills some neighborhood of $\LLzero\setminus\{\B0\}$.
\end{lemma}
\noindent
See \cite[Lemma 4.2]{blr1}.

There is another invariant line $\LLone:= \{t = 1\}$ passing through $\fptwo$ and $\fp$.  We have $\R|\LLone: z \rightarrow z^2$.

For the remainder of this section, it will convenient to use a system of affine 
coordinates centered at $\fp$.
We will use $(\one=Y/Z-T/Z,\two=T/Z)$, so that $\LLzero = \{\two = 0\}$ and $\LLone = \{\one = 0\}$. 
In these coordinates, 
\begin{equation}\label{r2}
(\one_{n+1},\two_{n+1}) = \left(\one_n^2\left(\frac{\one_n+2\two_n}{1+\two_n^2}\right)^2, \two_n^2\left(\frac{1+(\two_n+\one_n)^2}{1+\two_n^2}\right)^2\right) :=\R(\one_n,\two_n).
\end{equation}
As before, $\R|\LLzero: \one \rightarrow \one^4$ and $\R|\LLone: \two \rightarrow \two^2$.

\subsection{Co-dimension 1 B\"ottcher Function $\varphi$.}

We continue by exploring the some preliminary consequences of the hypothesis that $\stable$ is \ra~in such a full neighborhood of $\invcirc$.

\begin{proposition}\label{b} If $\stable$ is \ra~in a full neighborhood of $\invcirc$, then there is another neighborhood $\nbhd_0$ of $\invcirc$ and a holomorphic function $\varphi\colon\nbhd_0\rightarrow\mathbb C$ with
\begin{itemize}
\item[(i)] if $(\one,\two)\in\nbhd_0$ and $\R(\one,\two)\in\nbhd_0$, then $\varphi(\R(\one,\two))=\varphi(\one,\two)^4$,
\item[(ii)] $\stable=\{|\varphi(\one,\two)|=1\}$, and
\item[(iii)] $\varphi(\one,0)=\one$.
\end{itemize}
\end{proposition}

\begin{remark}
The function $\varphi$ is analogous to the one constructed in the Proof of Theorem A.  However, Proposition \ref{b} only gives that $\varphi$
is defined on a small neighborhood of $\invcirc$, which may not be forward invariant under $\R$.
\end{remark}

We will exploit the fact that each $x \in B$ is hyperbolic, emitting a stable manifold $\WW^s_{\loc}(x)$ that is a one-dimensional holomorphic curve transverse to $\LLzero$. Together, the union of stable manifolds of each $x \in B$ forms a foliation of $\stable$; see \cite[Proposition 9.2]{blr1}.

The notion of Levi-flat real-codimension 1 hypersurfaces $\Sigma\subset \mathbb{C}^n$ will be useful; for background see \cite{krantz,nish}. A $C^2$ hypersurface $\Sigma$ is Levi flat if though each point of $\Sigma$ there is a complex codimension 1, holomorphic hypersurface.  The union of these hypersurfaces is called the {\em Levi foliation of $\Sigma$}.  Thus, the preceding paragraph shows that $\stable$ is Levi flat.  Note that there is another, more common but equivalent, definition of Levi-flat given in terms of vanishing an appropriate Levi $(1,1)$-form \cite[page 126]{krantz}.

Rea's Theorem \cite{REA} holds in any codimension, but here we need only

\begin{rea}
Suppose $\Sigma$ is a Levi-flat, \ra\ hypersurface defined on some open $\nbhd_0\subset\mathbb C^n$.  Then there is a neighborhood $\nbhd\subset\nbhd_0$ of $\Sigma$ to which the Levi foliation extends uniquely and holomorphically.
\end{rea}

\noindent
We omit the proof, as it is rather simple in this case.
%
%

\begin{proof}[Proof of Proposition \ref{b}]
As stated above, $\stable$ is foliated by a family $\mathcal F$ of holomorphic stable curves at each point in $\invcirc$, so it's Levi flat.  Since $\stable$ is assumed to be \ra, Rea's Theorem implies that this Levi foliation extends to be a complex analytic foliation in a neighborhood of $\stable$.  Since the foliation $\mathcal F$ is transverse to $\LLzero$ at points of $\invcirc$, in a small enough neighborhood $\tilde\nbhd$, each curve $\gamma_{x}$ of the foliation is transverse to $\LLzero$.  Then we may assume $\nbhd$ is the union of connected components in $\tilde\nbhd$ of any leaf that intersects $\tilde\nbhd\cap\{\one=0\}$.  Let $\varphi\colon\nbhd\rightarrow\mathbb C$ be the map assigning to each $(\one,\two)\in\nbhd$ the point where $\gamma_{(\one,\two)}$ intersects $\two=0$.  From this, (ii) and (iii) follow immediately.  Note that it follows from a change of coordinates and the Implicit Function Theorem that $\varphi$ is holomorphic.

Define $\Omega_0$ to be the connected component of $\R^{-1}(\nbhd)\cap\nbhd$ containing $\invcirc$.  For each $\two_0$ with $|\two_0|$ sufficiently small, let $\mathcal L_{\two_0}:=\{\two=\two_0\}$.  Observe that $\invcirc_{\two_0}:=\stable\cap\mathcal L_{\two_0}$ is a topological circle.  Since $\invcirc_{\two_0}\subset\stable$, $(i)$ holds on $\invcirc_{\two_0}$ and, by uniqueness properties of holomorphic functions, it holds in some open neighborhood of $\invcirc_{\two_0}$ within $\mathcal L_{\two_0}$.  Varying $\two_0$, these neighborhoods form an open neighborhood of $\invcirc$ contained in $\nbhd_0$ on which $(i)$ holds.  This property then extends to all of $\nbhd_0$, since $\nbhd_0$ is connected.
\end{proof}

We can suppose that the domain $\Omega_0$ on which $\varphi$ is defined, given by Proposition
\ref{b}, is sufficiently small, so that it is contained in $\WW^s(\B0) \cup
\WW^s_{\loc}(\fptwo) \cup \stable \cup \WW^s_{\loc}(\fp)$. Since $\invcirc$ has a local product structure, it is isolated in the recurrent set.  Proof of this is similar to \cite[Proposition 4.4]{pujals2008two}.  Thus, we can choose $\nbhd_0$ smaller if necessary so that each orbit enters and leaves $\nbhd_0$ at most once.

\begin{proposition}\label{omega}
The domain $\nbhd_0$ may be extended to $\nbhd$, a neighborhood of $\LLzero\setminus\{\fp,\fptwo\}$, such that $\varphi\colon\nbhd\rightarrow\mathbb C$ is holomorphic,
\begin{itemize}
\item[(i)] If $(\one,\two)\in\nbhd$ and $\R(\one,\two)\in\nbhd$, then $\varphi(\R(\one,\two))=\varphi(\one,\two)^4$,
\item[(ii)] $\stable=\{|\varphi(\one,\two)|=1\}$, and 
\item[(iii)] $\varphi(\one,0)=\one$ for $x\in\LLzero\setminus\{\fp,\fptwo\}$.
\end{itemize}
\end{proposition}

In general, the push-forward of a function by a mapping is not well-defined.  However, if the mapping is proper, then it is well-defined by averaging over the fibers.  It was shown in \cite[Sec. 4.5]{blr1} that $\R$ has topological degree eight.  In the proof of Proposition \ref{omega} below, we mimic this push forward under a proper mapping.

\begin{proof}
Let $\nbhd_{n}:=\{x\colon\R^{-n}\{x\}\subseteq\Omega_{0}\}$ and $C_n$ be the critical value set for $\R^n$.  For $x\in\Omega_{n}\setminus C_n$, we may define
\begin{equation}\label{extension}
\varphi(x)=\frac{1}{8^n}\sum_{i=1}^{8^n}\varphi(y_i)^4,
\end{equation}
where $\{y_i\}_{i=1}^{8^n}=\R^{-n}\{x\}$.  Then locally about each $x\in\nbhd_n\setminus C_n$, $\varphi$ is holomorphic since each branch of $\R^{-n}$ is holomorphic by the Inverse Function Theorem. If $x$ follows a nontrivial loop around $C_n$, then $\varphi(x)$ has no monodromy since we are averaging over all of the fibers in (\ref{extension}).  Moreover, since $|\varphi|$ is bounded on $\nbhd_0$, (\ref{extension}) implies $|\varphi|$ is also bounded on $\nbhd_n\setminus C_n$.  Therefore, by the Riemann Extension Theorem, $\varphi$ can be extended through the critical value curves to be holomorphic on all of $\nbhd_{n}$.

If $x\in\nbhd_n\cap\nbhd_m$ with $n\geq m\geq0$, then $\R^{-n}\{x\},\R^{-m}\{x\}\subset\nbhd_0$.  
Since any orbit enters and leaves $\nbhd_0$ at most once, for any $y_i\in\R^{-m}\{x\}$ and each $z_j\in\R^{m-n}\{y_i\}$, we have that $z_j,\R(z_j),\dots,\R^{n-m}(z_j)=y_i\in\nbhd_0$.  Thus, $\varphi(y_i)=\varphi(\R^{n-m}(z_j))=\varphi(z_j)^{4^{n-m}}$ since (i) holds on $\nbhd_0$.  This implies \[\frac{1}{8^m}\sum_{y_i\in\R^{-m}(x)}\varphi(y_i)^{4^m}\ =\ \frac{1}{8^n}\sum_{z_j\in\R^{-n}(x)}\varphi(z_j)^{4^n},\] 
so that the two definition of $\varphi$ agree in $\nbhd_n\cap\nbhd_m$.

We obtain a well-defined holomorphic function $\varphi$ on 
\begin{equation}
\nbhd_{\infty}:=\bigcup_{n=0}^{\infty}\nbhd_n.
\end{equation}
Then we define $\nbhd$ to be the connected component of $\R^{-1}(\nbhd_{\infty})\cap\nbhd_{\infty}$ containing $\invcirc$.  Now (i) holds on all of $\nbhd$ using the exactly the same proof as in Proposition \ref{b}.i.

Since $\LLzero$ is forward invariant, $\nbhd_0$ intersects $\LLzero$, and
$\R|\LLzero$ is $\one\mapsto\one^4$, it follows that $\nbhd$ contains
$\LLzero\setminus\{\fp,\fptwo\}$.  The fact that
$\stable=\{|\varphi(\one,\two)|=1\}$ also follows from the fact that
$\nbhd_0\subset\nbhd$.
\end{proof}

\subsection{Local Properties Near $\fp$.}

In order to study the geometry of the extended domain $\Omega$ and the properties of $\varphi$, several technical results about the dynamics near $\fp$ will be required.  We may choose $\epsilon>0$ sufficiently small so that the bidisk
\begin{equation}
\bidisk_{\epsilon}:=\{|\one|<\epsilon,|\two|<\epsilon\},
\end{equation} 
is forward invariant, and $\R$ strictly decreases each component in modulus.  We continue by describing the trajectory of orbits as they converge to $\fp$.

\begin{proposition}\label{c} If $\epsilon>0$ is sufficiently small, then for any $\gamma\in\mathbb Z_{+}$, if $(\one_0,\two_0)\in\bidisk_{\epsilon}\setminus\LLzero$, then $|\one_n|/|\two_n|^{\gamma}\rightarrow0$.
\end{proposition}

This proposition implies that any point near $\fp$ and not on $\LL_0$ converges to $\fp$ with an arbitrarily high degree of tangency to $\LL_1$.

\begin{proof}
We first prove the proposition when $|\one_0|\leq|\two_0|^{\gamma}$.  Let $\blowup_n := \one_n/\two_n^{\gamma}$, so that
\begin{eqnarray}
\blowup_{n+1} = \frac{\one_{n+1}}{\two_{n+1}^{\gamma}} &=&  \frac{\one_n^2}{\two_n^{2\gamma}}\left(\frac{(1+\two_n^2)^{\gamma-1}(\one_n+2\two_n)}{(1+(\two_n+\one_n)^2)^{\gamma}}\right)^2 \nonumber \\  &=& \blowup_n^2 \two_n^2 \left(\frac{(1+\two_n^2)^{\gamma-1}(2+\blowup_n\two_n^{\gamma-1})}{(1+\two_n(1+\blowup_n\two_n^{\gamma})^2)^{\gamma}}\right)^2. \label{gammaeqn}
\end{eqnarray}
In the $(\two,\blowup)$ coordinates, $(0,0)$ is a superattracting fixed point for $\R$.  Then there is a $\delta>0$ such that any point with $|\two|,|\blowup|<\delta$ is in $\WW^s((0,0))$.  The closed disk $\{\two=0,|\blowup|\leq1\}$ collapses to $(0,0)$.  By continuity, there exists $\epsilon>0$ such that 
\begin{equation}
\R(\{|\two|<\epsilon,|\blowup|\leq1+\epsilon\})\ \subset\ \{|\two|,|\blowup|<\delta\}\ \subset\ \WW^s((0,0)).
\end{equation}
Thus, for  $(\one_0,\two_0)\in\bidisk_{\epsilon}$ with $\epsilon>0$ sufficiently small, if $|\one_0|\leq|\two_0|^{\gamma}$, then the result follows.

Now it suffices to show that if $\two_0\neq0$, then  there is some $N\geq0$ so that $|\one_n|\leq|\two_n|^{\gamma}$ for any $n\geq N$.  Let
\begin{equation}
M_1=\displaystyle\min_{(\one,\two)\in\overline\bidisk_{\epsilon}}\left|\frac{1+(\two+\one)^2}{1+\two^2}\right|^2\mbox{\ \ and \ }
M_2=\displaystyle\max_{(\one,\two)\in\overline\bidisk_{\epsilon}}9\left|\frac{1}{1+\two^2}\right|^2.
\end{equation}
As long as $|\one_n|\geq|\two_n|^{\gamma}$, we have 
\begin{equation}
|\two_{n+1}|\geq M_1|\two_n|^2\mbox{\   and \ }|\one_{n+1}|\leq M_2|\one_n|^{2+2/\gamma}.
\end{equation}
This implies that 
\begin{equation}
|\two_n|\geq A_1\rho_1^{2^n}\mbox{\ \ and \ }|\one_n|\leq A_2\rho_2^{(2+2/\gamma)^n}
\end{equation}
for some $A_i>0$ and $0<\rho_i<1$.  Then
\begin{equation}
\frac{|\one_n|}{|\two_n|^{\gamma}}\ \leq\ \frac{A_2}{A_1}\frac{\rho_2^{(2+2/\gamma)^n}}{\rho_1^{\gamma2^n}}=A\rho_2^{(2+2/\gamma)^n-a\gamma2^n}\rightarrow0,
\end{equation}
where $\rho_1=\rho_2^a$ and $A=A_2/A_1$.  Thus, for some iterate $m$, we have $|\one_m|\leq|\two_m|^{\gamma}$.
\end{proof}

Consider the ``bullet-shaped'' regions $B_{\gamma,c}:=\{(\one,\two)\colon|\one|\geq c|\two|^{\gamma}\}$, and let $B_{\gamma}\equiv B_{\gamma,1}$.  We will use the following horizontal and vertical cones:
\begin{equation}
\K^h:=\{|\two|\leq|\one|\}\	\mbox{ and }\	\K^v:=\{|\two|\geq|\one|\},
\end{equation}
noting that $\K^h=B_1$.

\begin{corollary}\label{bullets_backward_invariant} If $\epsilon>0$ is sufficiently small, then for any $\gamma\in\mathbb Z_{+}$, 
$\R^{-1}(B_{\gamma})\cap\bidisk_{\epsilon}\subset B_{\gamma}$.
\end{corollary}

\begin{corollary}\label{backwards_collapse} For any $\gamma\in\mathbb Z_{+}$, \label{p}$\displaystyle\bigcap_{n=0}^{\infty} \R^{-n}(B_{\gamma})\cap\bidisk_{\epsilon}=\LLzero\cap\bidisk_{\epsilon}$.
\end{corollary}

\begin{lemma}\label{nested_bullets_upper} For any sufficiently small $\epsilon>\sigma>0$ and any $\gamma\in\mathbb Z_{+}$, there exist $m\in\mathbb Z_{+}$ such that $\R^{-m}(B_{\gamma})\cap\left(\overline{\bidisk}_{\epsilon}\setminus\bidisk_{\sigma}\right)\subset\K^h$.
\end{lemma}

\begin{proof}
Consider the compact set $K:=\left(\overline{\bidisk}_{\epsilon}\setminus\bidisk_{\sigma}\right)\cap\K^v$.  It suffices to prove that there exists $m\in\mathbb Z_{+}$ such that $\R^m(K)\subset\bidisk_{\epsilon}\setminus B_{\gamma}$.  By the proof of Proposition 1.6, for each $x\in K$, there exists $m_x$ such that for any $m\geq m_x$, $\R^{m}x\in\bidisk_{\epsilon}\setminus B_{\gamma}$, which is open.  Then there is an open neighborhood $U_x$ of $x$ such that $\R^{m_x}(U_x)\subset\bidisk_{\epsilon}\setminus B_{\gamma}$.  Since $K$ is compact, there exists $m$ such that for any $x\in K$, $\R^m(x)\in\bidisk_{\epsilon}\setminus B_{\gamma}$.
\end{proof}

Recall that $\R|\LLzero$ is $\one\mapsto\one^4$ and $\R|\LLone$ is $\two\mapsto\two^2$.  The following distortion estimates allow local approximation of these properties near $\fp$. Also, recall the notation $(\one_n,\two_n)=\R^n(\one_0,\two_0)$.  Lastly, given two sequences $x_n$ and $y_n$, we will use $x_n\asymp y_n$ to mean that $a\leq|x_n/y_n|\leq A$ for some constants $0<a<A$.

\begin{proposition}\label{DISTORTION} For $\epsilon>0$ sufficiently small and any $\gamma\geq1$,
\begin{itemize}
\item[(i)] If $(\one_i,\two_i)\in B_{\gamma}\cap\bidisk_{\epsilon}$ for $i=0,\dots,n$, then $|\one_n|\asymp|\one_0|^{4^n}$.
\item[(ii)] If $(\one_i,\two_i)\in \bidisk_{\epsilon}\setminus B_{\gamma}$ for $i=0,\dots,n$, then $|\two_n|\asymp|\two_0|^{2^n}$.
\end{itemize}
\end{proposition}

\begin{proof}
Let
\begin{eqnarray}
A_i\ =\ \frac{1}{|\one_{n-i}|^2}\left|\frac{\one_{n-i}+2\two_{n-i}}{1+\two_{n-i}^2}\right|^2\ \leq\ 1+5\left|\frac{\two_{n-i}}{\one_{n-i}}\right|,
\end{eqnarray}
so that $|\one_{n-i+1}|=A_i|\one_{n-i}|^4$.  Inductively, we have
\begin{equation}\label{distortion constants}
|\one_n|\ =\ \left(\prod_{i=1}^nA_{i}^{4^{i-1}}\right)|\one_0|^{4^n}.
\end{equation}
Recall the constants $M_1\leq1\leq M_2$ from the proof of Proposition \ref{c}, which are independent of $\gamma$.  We have $|\two_{n}|\geq(M_1|\two_{n-i}|)^{2^i}$ and $|\one_{n}|\leq (M_2|\two_{n-i}|)^{(2+2/\gamma)^i}$, so it follows that 
\begin{eqnarray}
\left|\frac{\two_{n}}{\one_{n}}\right|\geq\frac{M_1^{2^i}}{M_2^{(2+2/\gamma)^i}}\left|\frac{\two_{n-i}}{\one_{n-i}^{(1+1/\gamma)^i}}\right|^{2^i}.
\end{eqnarray}
This implies there is a $0<\delta<1$ such that
\begin{eqnarray}
5\left|\frac{\two_{n-i}}{\one_{n-i}}\right|\ \leq\ 5(M_2|\one_{n-i}|)^{(1+1/\gamma)^i-1}\frac{M_2}{M_1}\left|\frac{\two_n}{\one_n}\right|^{1/2^i}\ \leq\ \delta^{(1+1/\gamma)^i},
\end{eqnarray}
since $M_2$ is a fixed constant, $|\one_{n-i}|<\epsilon$, and we can choose $\epsilon$ as small as we like.

It suffices to find uniform constants to estimate the product $\prod_{i=1}^{n}A_i^{4^{i-1}}$ independent $n$.  Observe 
\begin{eqnarray}
\prod_{i=1}^nA_{i}^{4^{i-1}}\ \leq\ \prod_{i=1}^n\left(1+5\left|\frac{\two_{n-i}}{\one_{n-i}}\right|\right)^{4^{i-1}}\ \leq\ \prod_{i=1}^{\infty}\left(1+\delta^{(1+1/\gamma)^i}\right)^{4^{i-1}},
\end{eqnarray}
where the last product converges since
\begin{eqnarray}
\sum_{i=1}^{\infty}4^{i-1}\log\left(1+\delta^{(1+1/\gamma)^i}\right)
\end{eqnarray}
converges.  Thus, there is a constant $A$ such that for any $n$, $\prod_{i=1}^nA_{i}^{4^{i-1}}\leq A$.   

A similar calculation can be done to find a uniform lower bound for the product.  Moreover, the proof for the vertical distortion control is similar (and easier).
\end{proof}

Consider $\R^{\ast}(\LLone)$, the pullback of the curve $\LLone=\{T=Y\}$, given by
\begin{equation}
-Z^2(T-Y)^2(T+Y)^2 = 0.
\end{equation}
The pullback of $\LLone$ contains $\LLone$, $\{Z=0\}$, and $\{T+Y=0\}$ (each counted with multiplicity two).  Call this last curve $D$, so in $(\one,\two)$ coordinates,
\begin{eqnarray}\label{q}
D&:=&\{\one+2\two=0\}.
\end{eqnarray}

\begin{figure}[ht]
\begin{center}
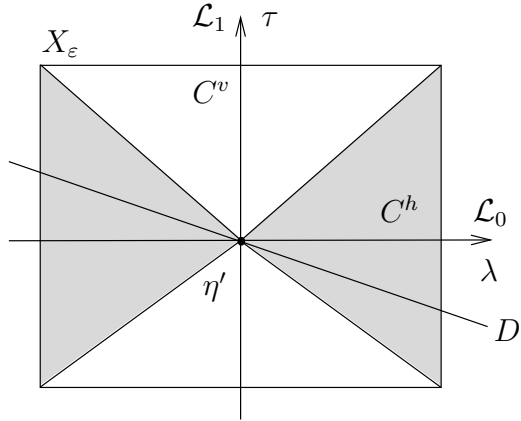
\end{center}
\caption{Bidisk neighborhood of $\fp$}%
\label{bidisk}
\end{figure}

\begin{lemma}\label{horizontal_history}  If $x\in\bidisk_{\epsilon}\setminus B_{3}$ and $\epsilon$ is sufficiently small, then  $\R^{-1}\{x\}\cap\K^h\neq\emptyset$ and $\R^{-1}\{x\}\cap\K^v\neq\emptyset$.
\end{lemma}

\begin{proof}
Let $N:=\{|\one|<\frac{1}{2}|\two|^2\}$, and note that if $x\in\bidisk_{\epsilon}\setminus B_3$, then $x\in N\cap\bidisk_{\epsilon}$.  Suppose $x\in N\cap\bidisk_{\epsilon}$ and let $(\one,\two)\in\R^{-1}\{x\}$.  Recall that the line $D:=\{\one+2\two=0\}$ has $\R(D)=\LLone$.  Also, note that $N$ is the union over $|c|\leq1/2$ of the curves $P_c:=\{\one=c\two^2\}$, and the preimage of any of these curves, $\R^{-1}(P_c)$, is the set of points satisfying
\begin{eqnarray}
\one^2\left(\frac{\one+2\two}{1+\two^2}\right)^2\ =\ c\two^4\left(\frac{1+(\one+\two)^2}{1+\two^2}\right)^4.
\end{eqnarray}
It follows that if $\epsilon>0$ is small enough that $\left|\sqrt c\frac{(1+(\one+\two)^2)^2}{1+\two^2}\right|\leq1$, then $\R^{-1}(P_c)$ is a set of points that satisfies
\begin{equation}\label{curve_bounds}
\left|\frac{\one}{\two}\right|\frac{|\one+2\two|}{|\two|}\ \leq\ 1.
\end{equation}
Since the curve $P_c$ is tangent to $\LLone$ and $\R(D\cup\LLone)=\LLone$, $\R^{-1}(P_c)$ must have a branch tangent to $\LLone$ and another branch tangent to $D$.  Moreover, by (\ref{curve_bounds}), these preimage curves must be contained in $\K^v$ and $\K^h$ respectively.  Thus, there is a preimage in $\K^h$ and another in $\K^v$.
\end{proof}


\begin{remark}
With a small amount of additional work, one can show that any point $x\in\bidisk_{\epsilon}$ with $\epsilon$ sufficiently small has a preimage
under the second iterate of $\R$ contained in $\K^h\cap\bidisk_{\epsilon}$.
\end{remark}

\begin{lemma}\label{nested_bullets} For any sufficiently small $\epsilon>0$ and any $k\in\mathbb Z_{+}$, there exist $\sigma>0$ and $\gamma\in\mathbb Z_{+}$ such that if $x\in\bidisk_{\sigma}\setminus B_{\gamma}$, then $x$ has a preorbit $\{\vseq i\}_{i=1}^k$ of length at least $k$ contained in $\K^v\cap\bidisk_{\epsilon}$.
\end{lemma}

\begin{proof}
Let $\R(\one,\two)=(\one',\two')\in\bidisk_{\sigma}\setminus B_{\gamma}$, so there is a  $\delta_1>0$ such that
\begin{equation}\label{gammadef} 
1\ \geq\ \frac{|\one'|}{|\two'|^{\gamma}}\ \geq\ \frac{|\one|^2}{|\two|^{2\gamma}}|\one+2\two|^2(1-\delta_1)^{2(\gamma-1)}.
\end{equation}
For large enough $\gamma$ and small enough $\sigma$, 
Lemma \ref{horizontal_history} implies there is some preimage $(\one,\two)\in \K^v$.  Then $|\two|\leq|2\two+\one|$, so
\begin{equation}
1\ \geq\ \frac{|\one|}{|\two|^{\gamma-1}}(1-\delta_1)^{\gamma-1}.
\end{equation}
There are $\delta_i$ for $i=2,\dots,\gamma-2$ so that after repeating this process, we have $\R^{\gamma-2}(\one_0,\two_0)\in\bidisk_{\epsilon}\setminus B_{\gamma}$ with
\begin{equation}\label{delta}
1\ \geq\ \frac{|\one_0|}{|\two_0|^{3}}(1-\delta_1)^{\frac{\gamma-1}{2^{\gamma-4}}}(1-\delta_2)^{\frac{\gamma-2}{2^{\gamma-3}}}\cdots(1-\delta_{\gamma-2})^{\frac{4}{2}}
(1-\delta_{\gamma-3})^{3}.
\end{equation}
Pick $\sigma$ small enough and $\gamma\geq k+3$ so that (\ref{delta}) implies $(\one_0,\two_0)\subset\K^v\cap\bidisk_{\sigma}$ and $\R^{-k}\{x\}\subset\bidisk_{\epsilon}$.
\end{proof}


\subsection{Properties of $\Omega$ and $\varphi$.}


\begin{figure}[ht]
\begin{center}
\scalebox{0.8}{
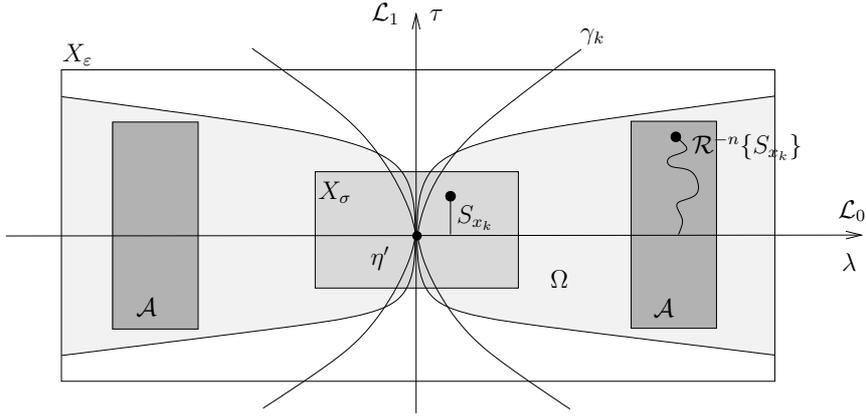
}
\end{center}
\caption{$\bidisk_{\sigma}$ (medium gray), $\mathcal A$ (dark gray), and $\nbhd$ (light gray); proportions have been modified to show detail.}%
\label{paths}
\end{figure}

\begin{lemma}\label{bullets_in_omega} For any $\gamma\in\mathbb Z_{+}$, there exists $\sigma>0$ such that $B_{\gamma}\cap\bidisk_{\sigma}\subset\Omega$.
\end{lemma}

\begin{proof}
By Proposition \ref{omega}, $\nbhd$ contains some neighborhood of
$\LLzero\setminus\{\fp,\fptwo\}$.  By Proposition \ref{DISTORTION}, there exists
$\epsilon>0$ sufficiently small so that for any $\gamma\in\mathbb Z_{+}$, the
horizontal distortion estimates can be applied in
$B_{\gamma}\cap\bidisk_{\epsilon}$.  Let $$\mathcal
A:=\{a\epsilon^{4^{j+2}}<|\one|<2A\epsilon^{4^j},|\two|<\delta\},$$ where $a$
and $A$ are the constants from the distortion estimate, $j\in\mathbb Z_{+}$ is
chosen so that $\mathcal A\subset\bidisk_{\epsilon}$, and $\delta<0$ is chosen
small enough so that $\mathcal A\subset\nbhd$.  See Figure \ref{paths}.

Let $x=(\one_0,\two_0)\in B_{\gamma}\cap\bidisk_{\sigma}$ and $S_x$ be the real straight line path connecting $x$ to $(\one_0,0)\in\LLzero$.  If $\sigma<\epsilon$ is sufficiently small, then by Corollary \ref{backwards_collapse} and the horizontal distortion estimates, there is an integer $n$ such that both $\R^{-n}\{S_x\},\R^{-n+1}\{S_x\}\subset\mathcal A$.  Then $S_x\subset\nbhd_{\infty}$ and $S_x\subset\R^{-1}(\nbhd_{\infty})$, and since $S_x$ is connected and intersects $(\LLzero\setminus\{\fp,\fptwo\})\subset\nbhd$, we have that $x\in S_x\subset\nbhd$.
\end{proof}


\begin{proposition}\label{phi_to_zero} For any sequence $\{x_m\}\subset\Omega$, if $x_m\rightarrow\fp$, then $\varphi(x_m)\rightarrow0$.
\end{proposition}

\begin{proof}
By Lemma \ref{bullets_in_omega}, there exists $\sigma>0$ such that
$B_{3}\cap\bidisk_{\sigma}\subset\nbhd$.  By the uniformity of $\varphi$ on
compact sets and the fact that $\varphi|\LLzero=id$, if $\delta>0$ small
enough, then $\mathcal A:=\{\sigma^{4^2}<|\one|<\sigma,|\two|<\delta\}\subset
B_3$, and $|\varphi(x)|<2\sigma$ for $x\in\mathcal A$.  By Lemma~\ref{horizontal_history}, there is a point in the preimage of each
$x_m\in\bidisk_{\sigma}\setminus B_{3}$ contained in $B_3$, and Corollary
\ref{bullets_backward_invariant}, $B_3$ is backward invariant.  Thus, there is
a backward orbit of each $x_m$ that remains in $B_3\subset\nbhd$.  Let
$\{x_{m,n}\}$ be this preorbit.  If $x_m$ sufficiently close to $\fp$, then by
Corollary \ref{backwards_collapse} there is an $N(m)$ such that
$x_{m,N(m)}\in\mathcal A$.  Using the invariance
$\varphi(\R^{n}(x))=\varphi(x)^{4^n}$, we have
\begin{equation}
|\varphi(x_m)|\ =\ |\varphi(x_{m,N})^{4^N}|\ <\ (2\sigma)^{4^N}.
\end{equation}
As $m$ goes to infinity, we need $N$ to go to infinity as well in order for $x_{m,N}$ to remain in $\mathcal A$.  This implies that the $\lim_{m\rightarrow\infty}|\varphi(x_m)|=0$.
\end{proof}


\subsection{Proof of Proposition \ref{u}.}

\begin{proposition}\label{d} For any $\epsilon>0$ sufficiently small, there is a sequence $\{\seq\}$ converging to $\fp$ such that for each $k$, $\seq$ has a preorbit of length $k$ contained in $\K^v\cap\bidisk_{\epsilon}$ and a preorbit of length $k$ contained in $\K^h\cap\bidisk_{\epsilon}$.  Moreover, any preimage of $\seq$ that is in $\bidisk_{\epsilon}$ is in $\nbhd$.  
\end{proposition}

\begin{proof}
By Lemma \ref{bullets_in_omega}, there exists $\epsilon>0$ sufficiently small so that $\bidisk_{\epsilon}\cap\K^h\subset\nbhd$.  For each $k\in\mathbb Z_{+}$, we do the following.  Using Lemma \ref{nested_bullets}, there exists $\gamma\in\mathbb Z_{+}$ and $\sigma>0$ such that $\seq\in\bidisk_{\sigma}\setminus B_{\gamma}$ has a preorbit $\vseq i\subset\K^v$ of length at least $k$.  Supposing that $\sigma$ is smaller if necessary, we can assure that $\R^{-k}\{\seq\}\subset\bidisk_{\epsilon}$.  Requiring that $\gamma\geq3$, Lemma \ref{horizontal_history} implies that $\seq$ has a first preimage, $\hseq 1$, in $\K^h$.  Since $\K^h$ is backward invariant by Corollary \ref{bullets_backward_invariant}, $\seq$ has a preorbit $\hseq i\subset\K^h$ of length at least~$k$.  

It remains to show that any preimage of $\seq$ that is in $\bidisk_{\epsilon}$ is in $\nbhd$.  First note that by  Lemma \ref{bullets_in_omega}, we can choose $\sigma$ smaller if necessary so that $\left(B_{\gamma+1}\cap\bidisk_{\sigma}\right)\subset\nbhd$.  By Lemma \ref{nested_bullets_upper}, there is an $m\in\mathbb Z_{+}$ such that $\R^{-m}(B_{\gamma+1})\cap(\bidisk_{\epsilon}\setminus\bidisk_{\sigma})\subset\K^h$.  Let $0<\tilde{\sigma}<\sigma$ be sufficiently small that 
if $x\in\bidisk_{\tilde{\sigma}}$, then $\R^{-m}\{x\}\subset\bidisk_{\sigma}$.  Let $\seq\in (B_{\gamma+1}\setminus B_{\gamma})\cap\bidisk_{\tilde{\sigma}}$.  Using that $B_{\gamma+1}$ is backward invariant, any preimage of $\seq$ that is in $\bidisk_{\sigma}$ will be in  $\left(B_{\gamma+1}\cap\bidisk_{\sigma}\right)\subset\nbhd$.  Meanwhile, by the choice of $\tilde{\sigma}$, any preimage that is in $\bidisk_{\epsilon}\setminus\bidisk_{\sigma}$ will be in $\bidisk_{\epsilon}\cap\K^h\subset\nbhd$.
\end{proof}


\begin{figure}[ht]
\begin{center}
\scalebox{0.85}{
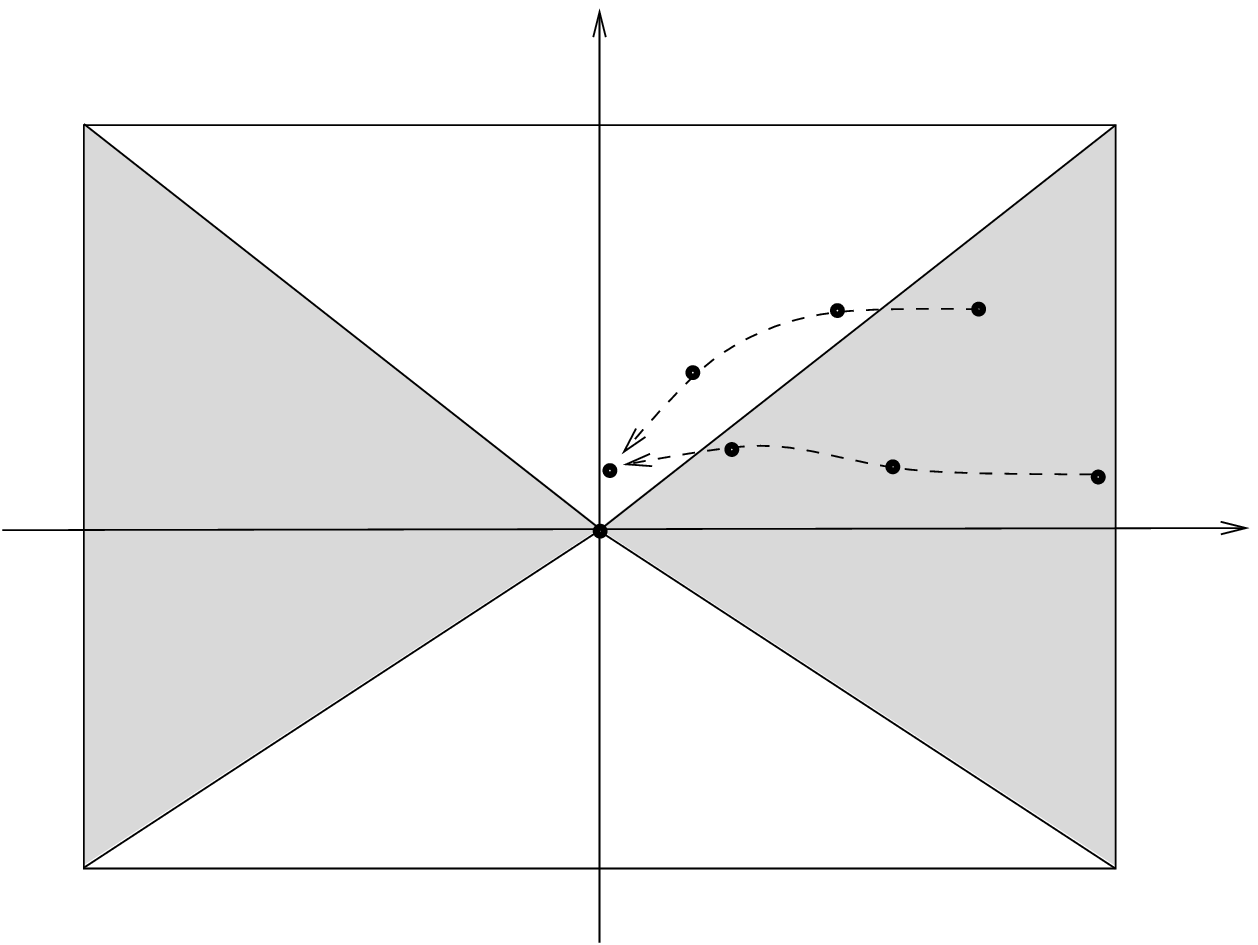
}
\end{center}
\caption{The preorbits $\{\vseq i\}$ and $\{\hseq i\}$}
\label{orbits}
\end{figure}

\begin{proof}[Proof of Proposition \ref{u}.]
Let $\{\seq\}\subset\Omega$ be a sequence as described in Proposition \ref{d}, and for each $k$, let $\{\vseq i\}_{i=1}^k\subset \K^v$ and $\{\hseq i\}_{i=1}^k\subset \K^h$ be preorbits of length $k$ such that $\vseq 0=\hseq 0=\seq$.  Each preorbit $\{\hseq i\}_{i=1}^k$ can be extended to a preorbit $\{\hseq{i}\}_{i=1}^{n(k)}$ with the element $\hseq{n(k)}$ being the  last preimage remaining in $\bidisk_{\epsilon}$.  See Figure \ref{orbits}.  Note that by Proposition \ref{d} for any $0\leq i\leq n(k)$, we have both $\vseq i,\hseq i\in\nbhd$.

We first show there is a subsequence of $\{\hseq{n(k)}\}$, that converges to a point in $\LLzero\setminus\{\fp,\fptwo\}$.  By construction, $\hseq{n(k)}$ is a preimage of $\hseq 1\in\K^h$, so

\begin{eqnarray}
\hseq{n(k)}\in\bigcap_{i=0}^{n(k)-1}\R^{-i}(\K^h)\cap\bidisk_{\epsilon}.
\end{eqnarray}
Also by construction, $\hseq{n(k)}\in\bidisk_{\epsilon}\setminus\R(\bidisk_{\epsilon})$, which has compact closure.  Thus, there is some subsequence such that $\hsubseq{n(k_j)}\rightarrow x_{\ast}$ with
\begin{eqnarray}
x_{\ast}\ \in\ \bigcap_{i=0}^{\infty}\R^{-i}(B_{\gamma_k})\cap\bidisk_{\epsilon}\ =\ \LLzero\cap\bidisk_{\epsilon}.
\end{eqnarray}
However, since each $\hseq{n(k)}\in\bidisk_{\epsilon}\setminus\R(\bidisk_{\epsilon})$, we must have $|x_{\ast}|\geq\epsilon^4$.

By the vertical and horizontal distortion distortion estimates in Proposition \ref{DISTORTION}, preimages of $\seq$ are escaping $\bidisk_{\epsilon}$ faster along $\hseq i$ than $\vseq i$, so we also have $\vseq{n(k)}\subset\bidisk_{\epsilon}$.  Note that $\vseq i$ may be in $\K^h$ for $k\leq i\leq n(k)$.  Then using both vertical and horizontal distortion, there is a constant $A$ so that
\begin{equation}
\dist(\vseq{n(k)},\fp)\ \leq\ A\ \dist(\seq,\fp)^{\frac{1}{2^k4^{n-k}}}\ \asymp\ A\ \dist(\hseq{n(k)},\fp)^{\frac{4^n}{2^k4^{n-k}}}\ \leq\ A\epsilon^{2^k},
\end{equation}
which converges to $0$ as $k\rightarrow\infty$.  Thus, the sequence $\vseq{n(k)}$ converges to $\fp$.

By Proposition \ref{phi_to_zero}, $\varphi(\vseq{n(k)})\rightarrow0$ as $k\rightarrow\infty$.  We also have that  $|\varphi(\hsubseq{n(k_j)})|\rightarrow|\varphi(x_{\ast})|\geq\epsilon^4$ as $k\rightarrow\infty$.  However, $\hsubseq{n(k_j)}$ and $\vsubseq{n(k_j)}$ are both $n$th preimages of $\subseq$, and using the invariance $\varphi(\R^n(x))=\varphi(x)^{4^n}$, this implies $|\varphi(\vsubseq{n(k)})|=|\varphi(\hsubseq{n(k)})|$ for every $n(k)$. Then $0=|\varphi(x_{\ast})|\geq\epsilon^4$, a contradiction.
\end{proof}


\subsection{Proof of Theorem B'}

\begin{lemma}\label{v} If $\stable$ is \ra~at $x \in \invcirc \setminus \{(\pm i,0)\}$, then $\stable$ is \ra~at $\R(x)$.
\end{lemma}

\begin{proof}
Images of \ra~hypersurfaces under holomorphic maps were considered by Baouendi
and Rothschild \cite{Baouendi_Rothschild}.  Suppose that $M$ is a germ of a
\ra~hypersurface in $\mathbb{C}^N$ and $H$ is the germ of a holomorphic map
from $\mathbb{C}^N$ to $\mathbb{C}^N$ with $H(0) = 0$.  The germ $H$ is called {\em finite}
if every point in some neighborhood of $0$ has finitely many preimages.  It is shown in \cite[Theorem 4]{Baouendi_Rothschild}
that if $H$ is finite and $M' := H(M)$ is smooth in some neighborhood of $0$, then $M'$ is actually \ra.

We are in the position to apply this result, since $\R$ sends $\stable$ from the neighborhood of any $x \in \invcirc$ 
to $\stable$ within a smaller neighborhood of $\R(x)$.  However, we must avoid the vertical lines $z = \pm i$, which are collapsed by $\R$ to the fixed point $(1,0) \in B$.  Away from these lines, $\R$
is finite.
\end{proof}


\begin{proof}[Proof of Theorem B'.]
By Proposition \ref{u}, there is some point $x \in \invcirc$ at which $\stable$ is not \ra.  We will now use the fact that $\R$ is expanding on
$\invcirc$ to show that $\stable$ is not \ra~in the neighborhood of any point of $\invcirc$.

Since $\R|\invcirc$ is $z \mapsto z^4$, it is
expanding on $\invcirc$, so there is some iterate $n$ such that $\R^n(U \cap
\invcirc) = \invcirc$.  Because we assumed $\stable$ is \ra~at every point of
$U \cap \invcirc$, we can use Lemma \ref{v} iteratively to see that $\stable$
is \ra~at every point of $\invcirc$, except perhaps at the iterated images of
$(\pm i,0)$.  However, these consist of just the fixed point $(1,0)$.  To see
that $\stable$ is \ra~at $(1,0)$ note that $(1,0)$ is also the image of $(-1,0)$
under $\R$, where $\stable$ is \ra.  Thus, $\stable$ must be \ra~at every point
of $\invcirc$, which is impossible by Proposition \ref{u}.

We now know that $\stable$ is not \ra~in the neighborhood of any point of
$\invcirc$.  However, it could still be \ra~in the neighborhood of some other
point.  We now show that this is also impossible.

Each stable manifold $\mathcal{W}^s_\loc(x_0)$ can be expressed as the graph of a convergent power series:
\begin{equation}\label{EQN:PS_FOLIATION}
z=h(t,z_0)=\sum_{j=0}^{\infty}a_j(z_0)t^j \qquad \mbox{where} \qquad x_0 = (z_0,0).
\end{equation}
Since each $\mathcal{W}^s_\loc(x_0)$ depends continuously on $z_0 \in \invcirc$, the coefficients $a_j(z_0)$ are continuous functions of $z_0$.
Therefore, there is a uniform radius of convergence $\delta > 0$.
For the remainder of the proof, we suppose that the neighborhood in which $\stable$ is defined is contained in $|t| < \delta/3$.

Suppose $\stable$ is \ra~in a neighborhood of some $x_1$.  
Then one can express leaves of the stable foliation near $x_1$ as graphs of some convergent power series
\begin{equation}\label{EQN:PS2}
z=k(t,z_1)=\sum_{j=0}^{\infty}b_j(z_1)(t-t_1)^j.
\end{equation}
The function $(z_1,t) \mapsto (z,t)$, with $z$ given by (\ref{EQN:PS2}), gives
a parameterization of $\stable$ near $x_1$ with $z_1$ varying over the real
analytic arc $\stable \cap \{t=t_1\}$ and $t$ varying over some complex disc
centered at $t_0$.   Since we have assumed $\stable$ is real analytic near
$x_1$, the parameterization is an analytic function.  In particular, 
$\frac{\partial^j}{\partial t^j} z$ is real analytic for each $j \geq 0$.  Restricting to 
$t=t_1$ we see that each of the coefficients $b_j(z_1)$ is a real
analytic function of $z_1$.

We now use this to show that $\stable$ is also \ra~in a neighborhood of the
unique point $x_0$ for which $x_1 \in \mathcal{W}^s_\loc(x_0)$. 
Since $\mathcal{W}^s_\loc(x_0)$ is the graph of a holomorphic function over $|t| < \delta$, $|t_1| < \delta/3$ implies that
(\ref{EQN:PS2}) converges on the disc $|t-t_0| < \delta/2$.  In particular, each of the holomorphic discs defined by (\ref{EQN:PS2}) crosses all the way
through $\invcirc$.  As they depend real analytically on $z_1$, this implies that $\stable$ is \ra~ in a neighborhood of $x_0 \in \invcirc$,
which is not possible.
\end{proof}


\section{Physical Interpretation.}\label{SEC:NON_ANALYTIC_LY}

In this section we will relate Theorems B' to the Ising Model on the
DHL.  We refer the reader to \cite{blr1,blr2} for physical background.
The DHL is a sequence of graphs $\Gamma_n$ obtained in a self-similar way.  Associated to each graph
is a partition function $\Zpart_n(z,t)$ whose zeros 
\begin{eqnarray*}
\mathcal{S}^c_n := \{(z,t) \in \C^2 \, : \, \Zpart_n(z,t)=0\}
\end{eqnarray*}
describe the singularities of the Ising model associated to $\Gamma_n$.  They are called the {\em Lee-Yang-Fisher zeros}.  
The actual physics is described by the limit $n \rightarrow \infty$.  It is proved in \cite{blr2} that the limiting distribution of zeros exists as a closed, positive $(1,1)$-current $\mathcal{S}^c$ on $\mathbb{P}^2$.  In fact, $\mathcal{S}^c = \frac{1}{2}\Psi^* S$, where $S$ is the Green current for $R$.  The support of $\mathcal{S}^c$ describes locus where phase transitions occur in $\C^2$.

It is shown in \cite{blr2} that at low complex temperatures ${\rm supp} \ \mathcal{S}^c$ coincides with $\stable$.  Combining Theorem B'
with the work from \cite{blr2} gives the following:
\begin{corollary}\label{COR:LYF}
At low complex temperatures ($|t|$ small), the locus of phase transitions for the Ising model on the DHL is a $3$ real-dimensional manifold that is $C^\infty$ but not real analytic.
\end{corollary}

A preferred subset of the Lee-Yang-Fisher zeros is obtained by requiring that $t
\in [0,1]$, which correspond to ``physical'' temperatures.  The Lee-Yang Circle
Theorem \cite{YL,LY} asserts that for each $n$ and fixed $t_0 \in [0,1]$, zeros
of partition function $\Zpart_n(z,t_0)$ corresponding to $\Gamma_n$ lie on
the unit circle
$\T_{t_0} := \{|z| =1,\, t=t_0\}$.
Let 
\begin{eqnarray*}
\mathcal{C} = \{|z| = 1, t \in [0,1]\}.
\end{eqnarray*}
The {\em Lee-Yang zeros} are defined by
\begin{eqnarray*}
\mathcal{S}_n := \{(z,t) \in \mathcal{C} \, : \, \Zpart_n(z,t)=0\}.
\end{eqnarray*}

Isakov \cite{ISAKOV} proved for any $t_0 >0$ sufficiently small the free energy for the Ising model on the $\mathbb{Z}^d$
lattice with $d > 1$ does not have analytic continuation through any point of the circle
$\mathbb{T}_{t_0}$.  This implies that the limiting distribution of Lee-Yang
zeros for the $\mathbb{Z}^d$ lattice with $d > 1$ does not have real
analytic density in the neighborhood of any point of the circle $t=t_0$.
In the remainder of this section, we discuss how Corollary \ref{COR:LYF} can be related to Isakov's result.

One can check that $\Rphys$ maps the Lee-Yang cylinder
$\mathcal{C}$ 
into itself, with the Lee-Yang zeros corresponding to $\Gamma_{n+1}$ obtained
by pulling back the Lee-Yang zeros corresponding to $\Gamma_n$ under $\R |
\Cphys$.  The map $\R\colon \mathcal{C} \rightarrow \mathcal{C}$ was also studied previously by Bleher and \v{Z}alys
\cite{BLEHER_ZALYS}.

In \cite{blr1}, Bleher, Lyubich, and Roeder describe the limiting distribution of Lee-Yang zeros for the DHL;
let us provide a very brief summary.  Let
$\Cphys_1 := \Cphys \setminus \{t=1\}$.  It was shown that $\mathcal{R} \colon
\Cphys_1 \rightarrow \Cphys_1$ is partially hyperbolic, with a unique central
foliation $\mathcal{F}^c$ which is vertical (with respect to a suitable cone
field) on $\Cphys_1$.  In particular, one can define the $\mathcal{F}^c$
holonomy map $g_{t}: \T_0 \rightarrow \T_{t}.$   The limiting distribution of
Lee-Yang zeros at temperature $t_0 \in [0,1)$ is obtained as the pushforward $\mu_t
= g_{{t_0}*} {\rm Leb}$, where ${\rm Leb}$ is the normalized Lebesgue measure on $\T_0$.

In a neighborhood of $\invcirc$, $\FF^c$ coincides with the stable foliation of
$\invcirc$, which is a union of the real analytic curves $\WW^s_\loc(x) \cap
\mathcal{C}$, taken over $x \in \invcirc$.  It is shown in \cite[Lemma
3.2]{blr2} that the stable foliation of $\invcirc$ within $\mathcal{C}$ has the
same regularity that the stable manifold $\stable$ does as a submanifold of
$\C^2$.  (In fact, $\stable$ was shown to be a $C^\infty$ manifold in \cite{blr2} by first
showing that the stable foliation of $\invcirc$ within $\mathcal{C}$ is $C^\infty$.)

Therefore, Theorem B' implies that the central foliation is not real analytic at low temperatures.  Moreover,
by \cite{blr1}, an open dense set of points from $\mathcal{C}$ have orbits converging to $\invcirc$.  Since $\mathcal{F}^c$ is invariant,
this implies the following:
\begin{theorem}
$\mathcal{F}^c$ is not real analytic in the neighborhood of any point of $\mathcal{C}$.
\end{theorem}
\noindent
Using the holonomy description of the limiting distribution of Lee-Yang zeros, we find the following modest analog of Isakov's Theorem for the DHL:

\begin{corollary}\label{COR_DENSE}
For any $z = e^{i \phi} \in \mathcal{B}$, there is a dense set of 
$t_0 \in [0,1]$ so that the limiting distribution of Lee-Yang zeros within $\mathbb{T}_{t_0}$ does not have
real analytic density at $(t_0 ,\phi)$.
\end{corollary}

\end{document}